\documentclass{www-notes}
\title{Small data global existence and decay for two dimensional wave maps}
\author{Willie Wai Yeung Wong\Affiliation{Michigan State University, East Lansing, USA; \url{wongwwy@math.msu.edu}}}
\keywords{AP, LinearWaves, HyperboloidalMethod, VectorFieldMethod}

\newcommand*\sobd{\mathfrak{s}} %

\begin{document}
\maketitle

\begin{wwwabstract}
	We prove the small-data global existence for the wave-map equation on $\Real^{1,2}$ using a variant of the vector field method. 
	The main innovations lie in the introduction of two new linear estimates. 
	First is the control of the dispersive decay of the solution $\phi$ itself (as opposed to its derivatives), via a logarithmic weighted Hardy inequality. This control has not been previously established using purely physical space methods in two spatial dimensions.  
	Second is a point-wise decay estimate for a twisted derivative of $\phi$ associated to the Morawetz $K$ multiplier, that cannot be reduced to point-wise decay estimates associated to the standard commutator vector fields. 
	As the linear theory is largely similar between dimensions, and in view of the novelty of the second innovation even in higher dimensions, we include a discussion of the method for $\Real^{1,d}$ in general.
	Both linear estimates are used crucially in our study: the control of the wave-map equation in the small data regime necessarily requires understanding the dispersive behavior of the bare solution $\phi$ itself, by virtue of the equation it satisfies. The point-wise decay for the twisted derivative allows us to avoid certain top-order logarithmic energy growths; this is indispensable for extending our argument from the case of compactly supported initial data (to which our methods are most naturally adapted) to initial data that are strongly localized but not necessarily of compact support, via an iterative construction. 
\end{wwwabstract}

\section{Introduction}

We consider the equation 
\begin{equation}\label{eq:WMmain}
	\Box_\eta \phi = \phi \cdot \eta(\D*\phi, \D*\phi) 
\end{equation}
which models the wave-map equation, in the small data regime. Here $\eta$ is the Minkowski metric on $\Real^{1,d}$, represented by the matrix $\diag(-1, 1, \ldots, 1)$ in rectangular coordinates, and so in standard coordinates we can rewrite \eqref{eq:WMmain} as
\[ 
	-\partial^2_{tt} \phi + \sum_{i = 1}^d \partial^2_{ii} \phi = \phi \big[- (\partial_t\phi)^2 + \sum_{i = 1}^d (\partial_i\phi)^2 \big].
\]
The general wave-map equation, in local coordinates, is the quasidiagonal system of semilinear wave equations for $\phi = (\phi^1, \ldots, \phi^N): \Real^{1,d} \to \Real^N$ given by the component-wise equations
\[ \Box_\eta \phi^C = \sum_{A,B = 1}^N \Gamma^C_{AB}(\phi) \eta(\D*\phi^A, \D*\phi^B)\]
where the Christoffel symbols $\Gamma$ captures the geometry of the target manifold. For solutions which are perturbations of a constant solution, we can perform a Taylor expansion of the functions $\Gamma$ and arrive at \eqref{eq:WMmain}, which captures the leading order contributions. 
In the small-data regime where decay can be proven, one can rather straightforwardly upgrade results concerning \eqref{eq:WMmain} to the full wave-map system. 
For more about the wave-map equation in general, see \cite{ShaStr1998}.

Using the dispersive decay of the linear wave equation, one can easily show, using the vector field method of Klainerman, that when the domain has spatial dimension $d \geq 3$ that solutions to \eqref{eq:WMmain} exist for all time and scatters provided that the initial data is sufficiently small in certain weighted energy norm. 
Indeed, multiplying \eqref{eq:WMmain} by $\partial_t \phi$ and integrating by parts on the domain $[0,T]\times \Real^d$, we see the energy inequality
\begin{equation}
	\frac12 \int_{\{T\}\times\Real^d} (\partial_t\phi)^2 + \abs{\nabla \phi}^2 \D{x} \leq \frac12 \int_{\{0\}\times \Real^d} (\partial_t\phi)^2 + \abs{\nabla\phi}^2  \D{x}+ \int_0^T \int_{\Real^d} \abs{\phi (\partial \phi)^3} \D{x} \D{t}.
\end{equation}
By Gronwall, we see then 
\begin{equation}
	\int_{\{T\}\times\Real^d} (\partial_t\phi)^2 + \abs{\nabla \phi}^2 \D{x} \leq \Big[ \int_{\{0\}\times \Real^d} (\partial_t\phi)^2 + \abs{\nabla\phi}^2 \D{x} \Big]\cdot \exp \Big[ C \int_0^T \sup_{\Real^d} \abs{\phi \partial\phi} \D{t} \Big].
\end{equation}
Using that the expected $L^\infty$ decay for the linear wave equation is at the rate $(1+t)^{(1-d)/2}$, we see that when $d \geq 3$ the integral within the exponential is expected to converge, giving global uniform bounds on the energy. 
This argument can be made precise via the Klainerman-Sobolev inequalities \cite{Klaine1985, Klaine1987} which allows us to relate boundedness of weighted energies with $L^\infty$ decay of the solution. 

In dimension $d = 2$ we see, however, that the expected $L^\infty$ decay for the linear wave equation is only at the rate $(1+t)^{-1/2}$, and so $\sup_{\Real^2} \abs{\phi\partial\phi}$ is not integrable in time. 
This difficulty can in principle be overcome by the fact that the nonlinearity $\eta(\D*\phi,\D*\phi)$ satisfies the null condition \cite{Klaine1986, Christ1986, Alinha2001, Alinha2001a}, which would imply that for all intents and purposes the term 
\[ \sup_{\Real^2} \abs{\phi\partial\phi} \lesssim (1+t)^{-3/2}.\]
The extra $t^{-1/2}$ decay upgrades the nonlinearity to be integrable in time. 

There is however, one other difficulty in dimension $d = 2$, and this relates to obtaining the decay estimate for $\phi$ itself. Classical Klainerman-Sobolev inequalities, when applied to the standard energy estimates, only control $\abs{\partial\phi}$ in $L^\infty$. 
One can na\"ively estimate $\abs{\phi(t)} \leq \abs{\phi(0)} + \int_0^t \abs{\partial\phi(s)} \D{s}$, but at a loss of one factor of $t$ decay (see also discussion in Remark 2, \S6.5 of \cite{Horman1997}).
A better method for dimensions $d \geq 3$ instead modifies the vector field method by using the Morawetz energy (where instead of multiplying by $\partial_t\phi$ and integrating by parts, one multiplies by $(t^2 + r^2)\partial_t \phi + 2tr \partial_r \phi + (d-1) t \phi$).
When $d \geq 3$, the resulting energy is coercive on $\norm[L^2(\Real^d)]{\phi}$, and direct applications of the classical vector field method yields also pointwise decay for the solution $\phi$ itself. This argument is worked out in detail in \cite{Klaine2001}.
When $d = 2$, however, the coercivity is lost, and while we have good control on $\eta(\D*\phi,\D*\phi)$ by virtue of the Klainerman-Sobolev inequalities and the null condition, we have no control on the $\phi$ factor in the nonlinearity in \eqref{eq:WMmain}.   

The goal of the present manuscript is to present a modified vector field method that allows us to recover, when $d = 2$, decay estimates on $\phi$ itself up to a logarithmic loss.
Through this control we are able to prove:
\begin{thm}
	Let $\phi_0, \phi_1\in C^\infty_0(B(0,1))$, and fix $k \geq 4$. There exists an $\epsilon > 0$ such that whenever $\norm[H^{k+1}]{\phi_0} + \norm[H^k]{\phi_1} < \epsilon$, there exists a global solution to \eqref{eq:WMmain} with $\phi(0,x) = \phi_0(x)$ and $\partial_t \phi(0,x) = \phi_1(x)$. Furthermore the solution decays to zero as $t\to +\infty$. 
\end{thm}
We will state and prove a more precise version of this theorem, including
the boundedness of weighted energies as well as the pointwise peeling estimates, in Section \ref{sect:2dWM}.
The above argument uses crucially that the initial data has bounded support.
For data that has unbounded support (but sufficient decay near infinity in a Sobolev sense), this can be overcome via an approximation procedure taking advantage of the scaling invariance of the wave-map equation. 
We discuss this in Section \ref{sect:noncompact}.
Crucial to this argument is a \emph{nonlinear stability} result for large-data ``dispersive'' solutions to the wave-map equation, in the spirit of \cite{Sideri1989}, which may be of independent interest and which we discuss in Section \ref{sect:largeDpert}. 

In closing, we note that the $d \geq 3$ versions of the analytic estimates presented below have prior analogues due to LeFloch and Ma \cite{LefMa2015}, and our main focus is in deriving the theory for $d = 2$. We include a discussion of the higher dimensional results to highlight the differences, as well as to showcase the minor notational advantage in our system which allows us to reduce the set of commutators used for analyzing the nonlinear problems. 

\subsection*{Acknowledgements}

The author would like to thank Pin Yu and Shiwu Yang for useful comments. 
Part of this research was conducted when the author was a visiting scholar at The Institute of Mathematical Sciences at The Chinese University of Hong Kong. The author would also like to express his gratitude to Zhouping Xin, Po Lam Yung, and The IMS for their hospitality during his visit.

\section{Hyperboloidal Global Sobolev Inequalities}

We begin by offering a streamlined proof of a family of weighted global Sobolev inequalities adapted to hyperboloids in Minkowski space.
Inequalities of this type was first introduced by Klainerman \cite{Klaine1985a,Klaine1993} for the study of the decay of the Klein-Gordon equation, and has been adapted by LeFloch and Ma \cite{LefMa2015} for treatment of coupled Klein-Gordon--Wave systems. 
The main novelty in this section of our work is the geometric formulation of our main inequalities as \emph{weighted Sobolev inequalities} adapted to the Lorentz boosts, in the sense that standard Sobolev inequalities on Euclidean domains are adapted to coordinate partial derivatives. 
In applications this means that \emph{the only commutator vector fields we will use are Lorentz boosts}, and in particular we can omit commutations with spatial rotations which are used in arguments based on Klainerman's original work. 

We restrict our attention to the interior of the future light-cone in Minkowski space, that is, the set $\{t > \abs{x}\}\subset \Real^{1,d}$. (We use $x^1, \ldots, x^d$ for the standard coordinates on $\Real^d$, and use $x^0$ or $t$ for the time coordinate.) 
Let
\begin{equation}
	\tau \eqdef \sqrt{t^2 - \abs{x}^2}
\end{equation}
and denote by $\Sigma_\tau$ its level sets. The $\Sigma_\tau$ are hyperboloids that asymptote to the forward light-cone centered at the origin.

Denote by $L^i$, $i\in \{1, \ldots,d\}$ the \emph{Lorentz boost} vector fields 
\begin{equation}
	L^i \eqdef t \partial_{x^i} + x^i \partial_t.
\end{equation}
By definition $L^i$ are tangent to the hypersurfaces $\Sigma_\tau$. The set $\{L^i\}$ is linearly independent and span the tangent space of $\Sigma_\tau$. We will also mention the rotational vector fields
\begin{equation}
	\Omega_{ij} \eqdef x^i \partial_{x^j} - x^j \partial_{x^i}.
\end{equation}
The $\Omega_{ij}$ are also tangent to $\Sigma_\tau$, and they admit the decomposition
\begin{equation}\label{eq:omegadecomp}
	\Omega_{ij} = \frac{x^i}{t} L^j - \frac{x^j}{t} L^i.
\end{equation}

The boosts and rotations form an algebra under commutation:
\begin{equation}\label{eq:algebra}
	\begin{gathered}
		[L^i, L^j] = \Omega_{ij}\\
		[\Omega_{ij},\Omega_{jk}] = \Omega_{ik}\\
		[L^i, \Omega_{ij}] = L^j
	\end{gathered}
\end{equation}

We can define a system of radial coordinates on $\{t > \abs{x}\}\setminus \{x = 0\}$. Let
\begin{equation}
	(\tau,\rho,\theta) \in \Real_+ \times \Real_+ \times \Sphere^{d-1}
\end{equation}
(where we identify $\Sphere^{d-1}$ canonically as the unit sphere in $\Real^d$) be the coordinate of the point
\begin{equation}\begin{gathered}
	t = \tau \cosh(\rho), \\
	x = \tau \sinh(\rho)\cdot \theta.
\end{gathered}\end{equation}
The Minkowski metric takes the warped product form 
\begin{equation}\label{eq:metricNew}
	\eta = - \D*{t}^2 + \sum_{i = 1}^d (\D*{x^i})^2 = - \D*{\tau}^2 + \tau^2 \D\rho^2 + \tau^2 \sinh(\rho)^2 \D\theta^2
\end{equation}
were by $\D*\theta^2$ we refer to the standard metric on $\Sphere^{d-1}$. Then clearly the induced Riemannian metrics on $\Sigma_{\tau}$, which we denote by $h_\tau$, have the coordinate expressions relative to the $(\rho,\theta)$ coordinates
\begin{equation}
	\begin{gathered}
		h_\tau = \tau^2 ( \D*\rho^2 + \sinh(\rho)^2 \D\theta^2) \\
		(h_\tau)^{-1} = \frac{1}{\tau^2} \big(\partial_\rho \otimes \partial_\rho + \frac{1}{\sinh(\rho)^2} \partial_\theta\otimes \partial_\theta\big)
	\end{gathered}
\end{equation}
where $\partial_\theta\otimes\partial_\theta$ is the inverse standard metric on $\Sphere^{d-1}$. 

One can check the identity 
\begin{equation}\label{eq:Lid1}
	\sum_{i = 1}^d L^i \otimes L^i = \partial_\rho \otimes \partial_\rho + \frac{\cosh(\rho)^2}{\sinh(\rho)^2} \partial_\theta\otimes\partial_\theta
\end{equation}
which implies that 
\begin{equation}\label{eq:Lid2}
	(\tau^{-2} h_\tau)^{-1} + \sum_{i < j} \Omega_{ij} \otimes \Omega_{ij} = \sum_{i = 1}^d L^i \otimes L^i
\end{equation}
as tensor fields along $\Sigma_\tau$. Note that $(\Sigma_\tau,\tau^{-2} h_{\tau})$ is isometric to the standard hyperbolic space $(\mathbb{H}^d,h)$. A particular consequence is the following coercivity property:
\begin{lem}\label{lem:Lcoercivity}
	Let $f$ be a smooth function on $\Sigma_\tau$. Then 
	\[  (h_{\tau}^{-1})(\D*f,\D*f) \leq \tau^{-2} \sum_{i = 1}^d \abs{L^i f}^2 .\]
\end{lem}

\begin{conv}
	Throughout the notation $\partial_t$ denotes the vector field associated to the $\partial_t$ partial differential in the standard coordinates of $\Real^{1,d}$; similarly $\partial_\tau$ and $\partial_\rho$ the vector fields associated to the corresponding partial derivatives in the $(\tau, \rho, \theta)$ coordinate system. 
\end{conv}

\begin{rmk}
	Observe that $\partial_\tau$ and $\partial_\rho$ have the decompositions
	\begin{gather*}
		\partial_\tau = \frac{1}{\tau} (t \partial_t + \sum_{i = 1}^d x^i \partial_{x^i}); \\
		\partial_\rho = t \partial_r + r \partial_t = \sum_{i = 1}^d \frac{x^i}{r} L^i.
	\end{gather*}
\end{rmk}

The global Sobolev inequalities that we use will be derived from the following standard Sobolev inequalities which we state without proof. For convenience we introduce the notation
\begin{equation}
	\sobd(d) \eqdef \lfloor \frac{d}{2} \rfloor + 1
\end{equation}
for the minimum integer satisfying $H^{\sobd(d)} \hookrightarrow L^\infty$. 
\begin{prop}[Standard Sobolev inequalities] \label{prop:StdSob}
	Below the symbol $\covD$ stands for the Levi-Civita connection on the corresponding Riemannian manifolds. 
	\begin{enumerate}
		\item Let $f$ be a function on the standard hyperbolic space $(\mathbb{H}^d,h)$, which we represent in polar coordinates $(\rho,\theta) \in \Real_+\times \Sphere^{d-1}$ as above.  Then
		\[ \sup_{\rho < \frac{5}{3}} \abs{f (\rho,\theta)}^2 \lesssim \sum_{k \leq \sobd(d)} \int_0^2 \int_{\Sphere^{d-1}} \abs[h]{ \covD^k f}^2 \sinh(\rho)^{d-1}  \D\theta \D\rho.\]
	\item Let $f$ be a function on the cylinder $\Real_+ \times \Sphere^{d-1}$ with the product metric. Then 
		\[ \sup_{\rho > \frac{4}{3}} \abs{ f(\rho,\theta)}^2 \lesssim \sum_{k \leq \sobd(d)} \int_1^\infty \int_{\Sphere^{d-1}} \abs{\covD^k f}^2 \D\theta \D\rho.\]
	\end{enumerate} 
\end{prop}

For convenience\footnote{We cannot use the standard multi-index notation, whose definition uses the fact that partial derivatives commute. Similarly, we cannot use the natural generalization of the multi-index notation as usually used in discussion of the vector field method, since there the set of first order operators used generate a Lie algebra (are closed under commutations). As we saw in \eqref{eq:algebra}, the Lorentz boost vector fields themselves are not closed under commutation.}, define the index set $\mathfrak{I}^m = \{1, \ldots, d\}^m$, and set
\[ \mathfrak{I}^{\leq m} \eqdef \bigcup_{0 \leq m' \leq m} \mathfrak{I}^{m'}; \quad \mathfrak{I} \eqdef \bigcup_{0 \leq m'} \mathfrak{I}^{m'}.\]
For $\alpha \in \mathfrak{I}$, its \emph{order} (denoted $\abs{\alpha}$), is the (unique) non-negative integer $m$ such that $\alpha \in \mathfrak{I}^m$. 
For $\alpha = (\alpha_1, \ldots, \alpha_m) \in \mathfrak{I}^m$, we denote by $L^\alpha$ the $m$\textsuperscript{th} order differential operator
\begin{equation}
	L^\alpha \phi \eqdef L^{\alpha_m} L^{\alpha_{m-1}} \cdots L^{\alpha_1} \phi.
\end{equation}
By convention, for the unique element $\alpha \in \mathfrak{I}^0$ we define the corresponding $L^\alpha \phi \eqdef \phi$. 

Putting everything together we have the following \emph{global} Sobolev inequality adapted to Lorentz boosts.
\begin{thm}[Global Sobolev Inequality] \label{thm:GlobSob}
	Let $\ell\in \Real$ be fixed. We have the following uniform estimate (the implicit constant depending on $d$ and $\ell$) for functions $f$ defined on the set $\{ t > \abs{x}\} \subset \Real^{1,d}$:
	\[ \abs{f(\tau,\rho,\theta)}^2 \lesssim \tau^{-d} \cosh(\rho)^{1-d-\ell} 
		\sum_{\alpha \in \mathfrak{I}^{\leq \sobd(d)}} \int_{\Sigma_\tau} \cosh(\rho)^\ell \abs{L^\alpha f}^2 ~ \underbrace{\tau^d \sinh(\rho)^{d-1} \D\theta \D\rho}_{\mathrm{dvol}_{\Sigma_\tau}}.
	\]
\end{thm}

\begin{proof}
	First we prove the estimate for $\rho < \frac53$: using that $(\Sigma_\tau , h_\tau)$ is conformal to $(\mathbb{H}^d, h)$ with $h_\tau = \tau^2 h$, we see that they have the same Levi-Civita connection. So the first part of \thref{prop:StdSob} implies 
	\[ \sup_{\rho < \frac53} \abs{f}^2 \lesssim \sum_{k \leq \sobd(d)} \int_{\Sigma_\tau \cap \{\rho < 2\}} \abs[\tau^{-2} h_\tau]{\covD^k f}^2 \mathrm{dvol}_{\tau^{-2} h_\tau}. \]
	Next, observe that for a fixed $(\rho,\theta)$, the higher derivative $\abs[\tau^{-2} h_\tau]{\covD^k L^i}$ is independent of $\tau$, and hence has universal bounds when $\rho < 2$. Observing that 
	\[ L^i \covD_a f = \covD_a L^i(f) - \covD_{\covD_a L^i} f \]
	we see after applying \thref{lem:Lcoercivity} repeatedly by induction that
	\[ \sum_{k \leq \sobd(d)} \abs[\tau^{-2}h_\tau]{\covD^k f}^2 \lesssim \sum_{\alpha\in \mathfrak{I}^{\leq \sobd(d)}} \abs{ L^\alpha f}^2 \]
	on the domain $\Sigma_\tau \cap \{\rho < 2\}$, for some universal constant depending only on the dimension $d$. Using that $\cosh(\rho)$ is bounded both above and below on the domain $\{\rho < 2\}$, we see our claim follows, with the $\tau^{-d}$ decay originating from $\tau^{-d} \mathrm{dvol}_{\Sigma_\tau} = \mathrm{dvol}_{\tau^{-2} h_\tau}$.

	Next we prove the estimate for $\rho > \frac43$: for this we apply the second part of \thref{prop:StdSob} to the function $f \cosh(\rho)^{\ell/2} \sinh(\rho)^{(d-1)/2}$. We note that when $\rho > 1$, both $\cosh(\rho)$ and $\sinh(\rho)$ are uniformly comparable to $e^\rho$, and so we have $(\partial_\rho)^j \cosh(\rho)^{\ell/2} \sinh(\rho)^{(d-1)/2} \approx \cosh(\rho)^{\ell/2} \sinh(\rho)^{(d-1)/2}$. Hence immediately we get
	\begin{equation}
		\sup_{\rho > \frac43}\big[ \abs{f}^2 \cosh(\rho)^\ell \sinh(\rho)^{d-1}\big] \lesssim
		\sum_{k \leq \sobd(d)} \int_1^\infty \cosh(\rho)^\ell \abs{\covD^k f}^2 \sinh(\rho)^{d-1} \D\theta \D\rho;
	\end{equation}
	we remark that the Levi-Civita connection and the norm are both taken with respect to the product metric on $\Real_+ \times \Sphere^{d-1}$. 
	Note that the inverse product metric has the decompsition 
	\[ \partial_\rho \otimes \partial_\rho + \sum_{i < j} \Omega_{ij} \otimes \Omega_{ij} \]
	and by \eqref{eq:omegadecomp} we have $\Omega \approx \frac{x}{t} L$, and we know that $\partial_\rho \approx \frac{x}{r} L$. Since the functions $x/r = \theta$ and $x/t = \theta \tanh(\rho)$ have bounded derivatives along $\Sigma_\tau$ (with respect to the product metric) to all orders away from $\rho = 0$, arguing as above we conclude that the quantity with respect to the product metric
	\[ \sum_{k \leq \sobd(d)} \abs{\covD^k f}^2 \lesssim \sum_{\alpha \in \mathfrak{I}^{\leq \sobd(d)}} \abs{L^i f}^2.\]
	And thus our claim follows. 
\end{proof}

\section{Decay of linear waves: $\partial_t$-energy}

In this section we apply \thref{thm:GlobSob} to get point-wise decay for the derivatives $\partial\phi$ of solutions to the linear wave equation $\Box \phi = 0$ on $\Real^{1,d}$. 
The main contribution of this paper is the explicit geometric decomposition \eqref{eq:tenergy} of the energy density of the linear wave equation along a hyperboloid. 
This makes explicit the fact that we see improved decay in ``good directions'' already at the level of first derivatives, and is compatible with our formulation of the global Sobolev inequality in Theorem \ref{thm:GlobSob}.

Let $\phi:\Real^{1,d} \to \Real$, define the standard energy quantity
\begin{equation}
	\big(\mathcal{E}_0[\phi]\big)^2 \eqdef \int_{\{0\} \times\Real^d} \abs{\partial_t\phi}^2 + \abs{\covD \phi}^2 \D{x}.
\end{equation}

\begin{prop} \label{prop:energy:derdecay}
	Let $\phi$ be a solution to $\Box\phi = 0$. Then we have the estimate in the region $\{t > \abs{x}\}$
	\begin{gather*}
		\abs{L^i \phi} \lesssim \underbrace{\tau^{1 - \frac{d}2} \cosh(\rho)^{1 - \frac{d}{2}}}_{t^{1-\frac{d}2}} \sum_{\alpha\in \mathfrak{I}^{\leq \sobd(d)}} \mathcal{E}_0[L^\alpha \phi], \\
		\abs{\partial_t \phi} \lesssim \underbrace{\tau^{-\frac{d}2} \cosh(\rho)^{1 - \frac{d}{2}}}_{t^{1 - \frac{d}{2}} (t+r)^{-1/2} (t-r)^{-1/2}} \sum_{\alpha\in \mathfrak{I}^{\leq \sobd(d)}} \mathcal{E}_0[L^\alpha \phi]. 
	\end{gather*}
\end{prop}

\begin{rmk}
	Note that the coefficients of the vector field $L^i$ in rectangular coordinates have size $\approx t$, so the first inequality really states that certain ``good derivatives'' (the ones in the span of $L^i$ with size $\approx 1/t$ coefficients) decay like $t^{-d/2}$, which captures exactly the expected improve decay over the naive $t^{-(d-1)/2}$. We also see from the second inequality that in the ``generic direction'' $\partial_t \phi$ decays only uniformly like $t^{-(d-1)/2}$ as expected; however, we also see improved \emph{interior} decay by $(t-r)^{-1/2}$.

	These peeling properties do not obviously hold using the classical vector field method, using the $\partial_t$-energy, for the lowest order derivatives of the solution. 
\end{rmk}

\begin{rmk}
	We do not need compact support for the estimates derived in this section. See also Footnote \ref{ftnt:cpt} below.

	An interesting feature of our decay estimates, when compared to the classical argument of Klainerman, is that we do not need to commute with the full set of generators of Poincar\'e group, nor do we need to commute with the scaling vector field $S = t \partial_t + \sum x^i \partial_{x^i}$. Here we only commute with the Lorentz boosts and not spatial rotations. 
\end{rmk}

\begin{proof}
	Let $Q_{ab}$ as usual denote the stress energy tensor $\partial_a \phi \partial_b\phi - \frac12 \eta_{ab} \eta(\D*\phi,\D*\phi)$. Consider the energy current 
	\[ \mathcal{J}_a = Q_{ab} (\partial_t)^b.\]
	It is standard that the space-time divergence $\covD^a \mathcal{J}_a = \Box \phi \partial_t \phi$, which in our case vanishes. 

	Integrating in the region between $\Sigma_\tau$ and $\{0\} \times \Real^d$, by the divergence theorem we have\footnote{Note that we don't necessarily have a conservation law; there is formally another boundary integral along null infinity that captures the energy radiated away. By the construction of the energy current, this quantity is signed and its omission in the formula below is what gives the $\leq$ sign instead of the $=$ sign.\label{ftnt:cpt}} 
	\[ \int_{\Sigma_\tau} Q(\partial_\tau, \partial_t) ~\mathrm{dvol}_{\Sigma_\tau} \leq \frac12 \mathcal{E}_0[\phi]^2. \]
	We can compute that 
	\[ \partial_t = \cosh(\rho) \partial_\tau - \tau^{-1} \sinh(\rho) \partial_\rho \]
	which implies 
	\[ Q(\partial_t,\partial_\tau) = \cosh(\rho) Q(\partial_\tau, \partial_\tau) - \tau^{-1} \sinh(\rho) Q(\partial_\rho, \partial_\tau).\]
	By orthogonality we have
	\[ Q(\partial_\rho, \partial_\tau) = \partial_\rho \phi \partial_\tau \phi;\]
	and a standard computation gives
	\[ Q(\partial_\tau, \partial_\tau) = \frac12 (\partial_\tau\phi)^2 + \frac1{2\tau^2} (\partial_\rho\phi)^2 + \frac{1}{2\tau^2 \sinh(\rho)^2} \abs{\partial_\theta \phi}^2.\]
	Combining the two and completing the square we get the following identity
	\[ 2 Q(\partial_\tau, \partial_t) = \frac{\cosh(\rho)}{\tau^2 \sinh(\rho)^2} \abs{\partial_\theta \phi}^2 + \frac{1}{\tau^2 \cosh(\rho)} (\partial_\rho \phi)^2 + \frac1{\cosh(\rho)} (\partial_t \phi)^2. \]
	By \eqref{eq:Lid1} we finally have
	\begin{equation}\label{eq:tenergy}
		Q(\partial_\tau, \partial_t) = \frac{1}{2\tau^2 \cosh(\rho)} \sum_{i = 1}^d (L^i \phi)^2 + \frac{1}{2\cosh(\rho)} (\partial_t\phi)^2.
	\end{equation}
	
	Now, as $L^i$ commutes with the wave operator we have that if $\phi$ solves $\Box \phi = 0$, so does $L^\alpha \phi$. This implies
	\begin{multline}\label{eq:tenergy2}
		\sum_{\alpha \in \mathfrak{I}^{\leq \sobd(d)}} \int_{\Sigma_{\tau}} \frac{1}{\tau^2 \cosh(\rho)} \sum_{i = 1}^d (L^i L^\alpha \phi)^2 \\
		+ \frac{1}{\cosh(\rho)} (\partial_t L^\alpha \phi)^2 ~\mathrm{dvol}_{\Sigma_\tau} 
		\leq \sum_{\alpha \in \mathfrak{I}^{\leq \sobd(d)}} \mathcal{E}_0[L^\alpha \phi]^2.
	\end{multline}
	The pointwise estimates for $L^i \phi$ now follows immediately from \thref{thm:GlobSob}.

	For $\partial_t \phi$, we need to control the integral of $(L^\alpha \partial_t \phi)^2$; the energy estimate only controls $(\partial_t L^\alpha \phi)^2$. We do so by computing the commutator:
	\[ [L^i, \partial_t ] = - \partial_{x^i}, \qquad [L^i, \partial_{x^i}] = - \delta_{ij} \partial_t.\]
	By induction this implies
	\[ \abs{L^\alpha \partial_t \phi} \lesssim \sum_{\abs{\beta} \leq \abs{\alpha}} \abs{\partial_t L^\beta \phi} + \sum_{\abs{\beta} \leq \abs{\alpha} - 1} \sum_{i = 1}^d \abs{\partial_{x^i} L^\beta \phi}.\]
	Noting that 
	\[ \partial_{x^i} = \frac{1}{t} (L^i - x^i \partial_t) \]
	we can bound
	\[ \sum_{\abs{\beta} \leq \abs{\alpha} - 1} \sum_{i = 1}^d \abs{\partial_{x^i} L^\beta \phi} \leq \sum_{\abs{\beta} \leq \abs{\alpha} - 1} \abs{\partial_t L^\beta \phi} + \frac{1}{\tau\cosh\rho} \sum_{\abs{\beta} \leq \abs{\alpha}} \abs{L^\beta \phi}.\]
	And therefore we have the estimate
	\begin{multline}
			\sum_{\alpha \in \mathfrak{I}^{\leq \sobd(d)}} \int_{\Sigma_{\tau}} \frac{1}{\tau^2 \cosh(\rho)} \sum_{i = 1}^d (L^i L^\alpha \phi)^2 + \frac{1}{\cosh(\rho)} (L^\alpha \partial_t \phi)^2 ~\mathrm{dvol}_{\Sigma_\tau} \\
			\lesssim \sum_{\alpha \in \mathfrak{I}^{\leq \sobd(d)}} \int_{\Sigma_{\tau}} \frac{1}{\tau^2 \cosh(\rho)} \sum_{i = 1}^d (L^i L^\alpha \phi)^2 + \frac{1}{\cosh(\rho)} (\partial_t L^\alpha \phi)^2 ~\mathrm{dvol}_{\Sigma_\tau}.
	\end{multline}
	Applying \thref{thm:GlobSob} we also get the decay estimate for $\partial_t\phi$. 
\end{proof}

\section{Decay of linear waves: $K$-energy}

The Morawetz energy can also be used with the hyperboloidal foliation; this gives improved decay properties. As in the previous section, the main contribution is the identity \eqref{eq:Kenergy}. 
For convenience, we shall assume that the data has compact support in this section. 

\begin{prop}\label{prop:ImprovedDecay}
	Let $\phi$ be a solution to $\Box \phi = 0$. Suppose $\phi(2,x)$ and $\partial_t\phi(2,x)$ are both supported on the ball of radius 1 centered at the origin. Then in the region $t > \max(2,|x|)$ we have
	\[
		\abs{L^i \phi} \lesssim \tau^{-\frac{d}2} \cosh(\rho)^{1 - \frac{d}{2}} \left( \norm[H^{\sobd(d)}]{\partial_t\phi(2,\text{---})} + \norm[H^{\sobd(d)+1}]{\phi(2,\text{---})}\right).
	\]
\end{prop}

\begin{rmk}
	We see that this gives a gain of an additional factor of $\tau^{-1}$ decay for derivatives in the ``good directions''. In terms of uniform-in-$x$ decay, this means a gain of $t^{-1/2}$ (as well as a gain of $(t - r)^{-1/2}$ which is not uniform on constant $t$ hyperplanes). 
\end{rmk}

\begin{proof}
	Define the modified current
	\begin{equation}
		\mathcal{K}_a = Q_{ab} K^b + \frac{d-1}{2} t \partial_a (\phi^2) - \frac{d-1}{2} \phi^2 \partial_a t
	\end{equation}
	where $K$ is the vector field
	\begin{equation}
		K \eqdef (t^2 + |x|^2) \partial_t + 2tr \partial_r = \tau^2 \cosh(\rho) \partial_\tau + \tau\sinh(\rho) \partial_\rho.
	\end{equation}
	It is a simple exercise to check that the space-time divergence
	\begin{equation}
		\covD^a \mathcal{K}_a = \Box\phi [ K\phi + (d-1) t \phi].
	\end{equation}
	By virtue of finite speed of propagation and the divergence theorem, we have that
	\[ \int_{\Sigma_\tau} \mathcal{K}_a (\partial_\tau)^a ~\mathrm{dvol}_{\Sigma_\tau} = \int_{\{2\}\times\mathbb{R}^d} \mathcal{K}_a (\partial_t)^a \D{x}.\]
	The integral on the right is entirely determined by the initial data, so we focus on evaluating the integral over $\Sigma_\tau$. The integrand is
	\[ \mathcal{K}_a(\partial_\tau)^a = Q(K, \partial_\tau) + \frac{d-1}{2} \tau \cosh(\rho) \partial_\tau (\phi^2) - \frac{d-1}{2} \phi^2 \cosh(\rho). \]
	The middle term, observe, can be re-written in terms of $K$ and $\partial_\rho$.
	\begin{align*}
		\mathcal{K}_a(\partial_\tau)^a &= Q(K, \partial_\tau) + \frac{d-1}{2\tau} K(\phi^2) - \frac{d-1}{2} \sinh(\rho) \partial_\rho (\phi^2) - \frac{d-1}{2} \phi^2 \cosh(\rho)\\
		& = Q(K, \partial_\tau) + \frac{d-1}{2\tau} K(\phi^2) - \frac{d-1}{2} \partial_\rho [\sinh(\rho) \phi^2].
	\end{align*}
	We wish to integrate
	\begin{multline*}
		\int_{\Sigma_\tau} \mathcal{K}_a (\partial_\tau)^a ~\mathrm{dvol}_{\Sigma_\tau} = \\
		\int_0^\infty \int_{\Sphere^{d-1}} \left\{ Q(K, \partial_\tau) + \frac{d-1}{2\tau} K(\phi^2) - \frac{d-1}{2} \partial_\rho[\sinh(\rho) \phi^2] \right\} \tau^d \sinh(\rho)^{d-1} \D\theta \D\rho.
	\end{multline*}
	Using the finite-speed of propagation property again, which asserts that $\phi$ has compact support on each $\Sigma_\tau$, we can integrate the final term in the braces by parts to get
	\begin{multline*}
		\int_{\Sigma_\tau} \mathcal{K}_a (\partial_\tau)^a ~\mathrm{dvol}_{\Sigma_\tau} = \\
		\int_0^\infty \int_{\Sphere^{d-1}} \left\{ Q(K, \partial_\tau) + \frac{d-1}{2\tau} K(\phi^2) + \frac{(d-1)^2}{2} \cosh(\rho) \phi^2 \right\} \tau^d \sinh(\rho)^{d-1} \D\theta \D\rho.
	\end{multline*}

	Next, we can write
	\begin{align*}
		Q(K,\partial_\tau) & = \tau^2 \cosh(\rho) Q(\partial_\tau, \partial_\tau) + \tau \sinh(\rho) Q(\partial_\tau, \partial_\rho) \\
		&= \frac12 \cosh(\rho) \left[ \tau^2 (\partial_\tau\phi)^2 + (\partial_\rho\phi)^2 + \frac{1}{\sinh(\rho)^2} \abs{\partial_\theta \phi}^2 \right] + \tau \sinh(\rho) \partial_\tau \phi \partial_\rho \phi\\
		& = \frac1{2\cosh(\rho)} \Big\{ \big[ \tau\cosh(\rho) \partial_\tau \phi + \sinh(\rho) \partial_\rho \phi\big]^2 + (\partial_\rho\phi)^2 + \frac{\cosh(\rho)^2}{\sinh(\rho)^2} \abs{\partial_\theta\phi}^2 \Big\}\\
		& = \frac{1}{2\cosh(\rho)} \Big[ \frac{1}{\tau^2} (K\phi)^2 + \sum_{i = 1}^d (L^i \phi)^2 \Big].
	\end{align*}
	So finally, completing the square one more time we get
	\begin{multline}\label{eq:Kenergy}
		\int_{\Sigma_\tau} \mathcal{K}_a (\partial_\tau)^a ~\mathrm{dvol}_{\Sigma_\tau} = \int_{\Sigma_\tau} \frac{1}{2\cosh(\rho)} \sum_{i = 1}^d (L^i \phi)^2 \\
		+ \frac{1}{2\tau^2 \cosh(\rho)} \big[ K\phi + (d-1) t \phi \big]^2~ \mathrm{dvol}_{\Sigma_{\tau}}.
	\end{multline}

	Applying this estimate to $L^\alpha \phi$ for $\alpha \in \mathfrak{I}^{\leq \sobd(d)}$, we see that 
	\[ \sum_{\alpha \in \mathfrak{I}^{\leq \sobd(d)}} \int_{\Sigma_\tau} \frac{1}{\cosh(\rho)} \sum_{i = 1}^d \abs{L^i L^\alpha \phi}^2  ~\mathrm{dvol}_{\Sigma_\tau} \]
	is uniformly bounded by the initial data, which in turn is bounded by 
	\[ \norm[H^{\sobd(d)}]{\partial_t\phi(2,\text{---})}^2 + \norm[H^{\sobd(d)+1}]{\phi(2,\text{---})}^2 \]
	(we make use of the compact support again to control the various weights). One application of \thref{thm:GlobSob} gives the desired decay.  
\end{proof}

Comparing the energy quantities \eqref{eq:Kenergy} and \eqref{eq:tenergy}, we see a clearly hierarchy. This is related to the $r^p$-weighted energy estimates of Dafermos and Rodnianski \cite{DafRod2009}. 
To make this comparison more explicit, we show that using the estimate \eqref{eq:Kenergy} one can also derive decay estimates for the twisted derivative $K\phi + (d-1)t\phi$. 
Observe first the commutator identity
\begin{equation}
	[L^i, K + (d-1) t] = \frac{x^i}{t} (K + (d-1)t) + \frac{\tau^2}{t} L^i
\end{equation}
which implies that 
\begin{multline}\label{eq:Kcommu}
	\int \frac{1}{\tau^2 \cosh(\rho)} \abs{L^i [K + (d-1)t]\phi}^2 ~\mathrm{dvol}_{\Sigma_{\tau}} \\
	\lesssim 
	\int \frac{1}{\tau^2 \cosh(\rho)} \left( \abs{[K + (d-1)t] L^i \phi}^2 + \abs{[K + (d-1)t]\phi}^2 \right) + \frac{1}{\cosh(\rho)} (L^i\phi)^2 ~\mathrm{dvol}_{\Sigma_{\tau}}. 
\end{multline}
Therefore applying this to $L^\alpha \phi$, we see that the energy identity shows that
\[ \sum_{\alpha \in \mathfrak{I}^{\leq \sobd(d)}} \int_{\Sigma_\tau} \frac{1}{\cosh(\rho)} \sum_{i = 1}^d \abs{L^\alpha [K + (d-1)t] \phi}^2  ~\mathrm{dvol}_{\Sigma_\tau} \]
is uniformly bounded by the $K$-energies for up to $\sobd(d)$ derivatives of the solution, which after an application of \thref{thm:GlobSob} gives that 
\begin{equation}\label{eq:Kptwbd}
	|K\phi + (d-1)t\phi| \lesssim \tau^{1 - \frac{d}{2}} \cosh(\rho)^{1 - \frac{d}{2}}.
\end{equation}

\begin{rmk}
	The point-wise bound \eqref{eq:Kptwbd} is, to the author's knowledge, new. 
	Observe that since $K$ decomposes as $K = \tau^2 \partial_t + 2 \sum_{i = 1}^d x^i L^i$, this estimate \emph{cannot} be obtained from \thref{prop:ImprovedDecay} and \thref{prop:energy:derdecay}.
	The twisted structure involving the lower order term $(d-1)t\phi$ is important as this imparts additional cancellations. Indeed, the pointwise estimates for $\partial_t\phi$, $L^i\phi$, and $(K + (d-1)t)\phi$ are independent of each other in the sense that any two of the three is not sufficient to derive the third. 
\end{rmk}

\begin{rmk}
	This point-wise bound \eqref{eq:Kptwbd} will play an important role in the discussion in the sequel. When applying the $K$-energy to nonlinear applications, one pays for the additional decay gained by requiring stronger estimates on the inhomogeneities. 
	In traditional applications of the vector field method this often translates to a logarithmic growth of the top-order energies, when certain terms most naturally estimated using $(K+(d-1)t)\phi$ are instead estimated using the decomposition described in the previous remark. 
	We will take full advantage of this estimate by performing the decomposition \eqref{eq:bilinearest} and fully expressing the nonlinearities, where necessary, using this operator. 
	This allows us to avoid the log-loss of top-order energies, which is then itself crucial for proving our perturbation \thref{thm:pert}. 
\end{rmk}

\section{Decay of $\phi$ itself}

In this section we prove that the energy integrals considered in the previous two sections given in \eqref{eq:tenergy} and \eqref{eq:Kenergy} have nice coercivity properties on the $L^2$-integral for $\phi$ itself. 
In the context of the study of wave equations, our results highlight the following two ideas:
\begin{enumerate}
	\item First, it shows that the Morawetz energy is not necessary for controlling the decay of the solution $\phi$ itself. This has applications to situations where the Morwawetz energy is not available, due to the lack of conformal symmetry of the underlying equation. A particular such application is to the study of wave equations on Kaluza-Klein backgrounds. See the discussion in \thref{rmk:KK}. 
	\item Second, it shows that one can in fact obtain estimates on the solution $\phi$ itself using purely physical space techniques, in dimension $d = 2$. Such estimates were previously unavailable. The downside to our argument is that compact support of initial data seems essential in this case. 
\end{enumerate}

The basic technical tool for obtaining the coercivity are the following \emph{Hardy} inequalities. 
\begin{lem}[Hardy, $d \geq 3$] \label{lem:hardy:3}
	Let $d \geq 3$, then
	\[ \int_{\Sigma_\tau} \frac{1}{\cosh(\rho)} \phi^2 ~\mathrm{dvol}_{\Sigma_\tau} \leq \frac{4}{(d-2)^2} \int_{\Sigma_\tau} \frac{1}{\cosh(\rho)} \sum_{i = 1}^d (L^i \phi)^2 ~\mathrm{dvol}_{\Sigma_\tau}.\]
\end{lem}

\begin{lem}[Hardy, $d = 2$] \label{lem:hardy:2}
	When $d = 2$, then 
	\[ \int_{\Sigma_\tau} \frac{1}{\cosh(\rho)} \phi^2 ~\mathrm{dvol}_{\Sigma_\tau} \lesssim \int_{\Sigma_\tau} \frac{(1 + \rho^2)}{\cosh(\rho)} \sum_{i = 1}^d (L^i \phi)^2 ~\mathrm{dvol}_{\Sigma_\tau}.\]
\end{lem}

\begin{rmk}
	When $d \geq 3$, noting that $\sinh(\rho)^{d-1} / \cosh(\rho)$ grows asymptotically like $\sinh(\rho)^{d-2}$, we see that our Hardy inequality can be viewed as the statement that the hyperbolic space $\mathbb{H}^{d-1}$ of one lower dimension has a spectral gap (and hence $\mathring{H}^1$ controls $L^2$). When $d = 2$ this obviously degenerates, and the best we can have is essentially the Hardy inequality on the half-line. Note that compared to the exponential weight in $\rho$, the polynomial $\rho^2$ is effectively a logarithm. 
\end{rmk}

\begin{rmk}
	The Hardy type inequalities of Lemma \ref{lem:hardy:3} is well-known in the study of weighted Sobolev spaces. See, e.g.\ Theorem 1.3 in \cite{Bartni1986}; the $d = 2$ case corresponds to the forbidden $\delta = 0$ case in that theorem. Our main contribution in this section is Lemma \ref{lem:hardy:2}. We include the proof of both for completeness. 
\end{rmk}

\begin{proof}[\thref{lem:hardy:3}]
	For any $f:\Real_+ \to \Real$ that decays rapidly as $\rho \to \infty$, observe that
	\[ \int_0^\infty \partial_\rho \big[ \sinh(\rho)^\alpha f(\rho)^2\big] \D\rho = 0\]
	by the fundamental theorem of calculus. So we have
	\[ \alpha \int_0^\infty f(\rho)^2 \cosh(\rho) \sinh(\rho)^{\alpha - 1} \D\rho \leq \Big| 2 \int_0^\infty f(\rho) f'(\rho) \sinh(\rho)^\alpha \D{\rho} \Big|.\]
	Applying Cauchy-Schwarz to the right, and squaring both sides, give
	\[ \alpha^2 \int_0^\infty f(\rho)^2 \cosh(\rho) \sinh(\rho)^{\alpha - 1} \D\rho \leq 4 \int_0^\infty [f'(\rho)]^2 \frac{\sinh(\rho)^{\alpha + 1}}{\cosh(\rho)} \D{\rho}.\]
	Now set $\alpha = d-2$, so $\alpha + 1 = d-1$. We can integrating in the spherical directions and use \eqref{eq:Lid1} to control $(\partial_\rho f)^2$. On the left we apply the simple observation that $\tanh(\rho) \leq 1$. This leads to exactly the claimed inequality. 
\end{proof}

\begin{proof}[\thref{lem:hardy:2}]
	For any $f:\Real_+ \to \Real$ that decays rapidly as $\rho \to \infty$, we observe that
	\[ \int_0^\infty \partial_\rho \big[ \rho f(\rho)^2 \big] \D\rho = 0\]
	which implies
	\begin{multline*}
		\int_0^\infty f(\rho)^2 \D\rho \leq 2 \int_0^\infty f(\rho) f'(\rho) \rho \D(\rho) \\
			\leq 2 \Big[ \int_0^\infty \frac{\rho}{\sqrt{1 + \rho^2}} f(\rho)^2 \D\rho \Big]^{\frac12} \cdot \Big[ \int_0^\infty \rho \sqrt{1 + \rho^2} [f'(\rho)]^2 \D\rho \Big]^{\frac12}
	\end{multline*}
	by Cauchy-Schwarz. This implies, since $\rho / \sqrt{1 + \rho^2} < 1$, 
	\[ \int_0^\infty f(\rho)^2 \D\rho \leq 4 \int_0^\infty [f'(\rho)]^2 \rho \sqrt{1 + \rho^2}  \D\rho. \]
	Now since $\lim_{\rho \to 0} \tan(\rho) / \rho = 1$, there exists some constant $C$ such that $\rho \leq C\tan(\rho) \sqrt{1 + \rho}$ for all $\rho > 0$. 
	This implies 
	\[ \int_0^\infty f(\rho)^2 \tanh(\rho) \D\rho  \leq C \int_0^\infty (1 + \rho^2) \tanh(\rho) [f'(\rho)]^2 \D\rho \]
	which, after integrating in the spherical directions and using \eqref{eq:Lid1} is exactly the claimed inequality. 
\end{proof}

Applying the Hardy inequalities to the energy estimates \eqref{eq:tenergy2} and \eqref{eq:Kenergy} we get the following decay estimates for $\phi$ itself when dimension $d \geq 3$.  
We omit the obvious proofs. 
\begin{prop}[$d \geq 3$, $\partial_t$-energy]
	Let $\phi$ solve the linear wave equation on $\Real^{1,d}$ with finite energy initial data, in the sense that $\sum_{\alpha \in \mathfrak{I}^{\leq \sobd(d)}} \mathcal{E}_0[L^\alpha \phi]^2 < \infty$. Then 
	\[ \abs{\phi} \lesssim t^{1 - d/2}.\]
\end{prop}
\begin{prop}[$d \geq 3$, $K$-energy]
	Let $\phi$ solve the linear wave equation on $\Real^{1,d}$ with compactly supported initial data (e.g.\ satisfying the hypotheses of \thref{prop:ImprovedDecay}), then
	\[ \abs{\phi} \lesssim \frac{1}{t^{\frac{d}2 - 1} \sqrt{(t+r)(t-r)}}. \]
\end{prop}

\begin{rmk}
	We note that the estimates are sharp. The finite $\partial_t$-energy condition is compatible with initial data $\phi(0,x) = 0$ and $\partial_t \phi(0,x) = (1 + r^2)^{-d/4 - \epsilon}$. One can check using the fundamental solution that the solution $\phi(t,0)$ for this data decays like $t^{1 -d/2 - 2\epsilon}$. 

	For the case with the $K$-energy, we note that the rate $t^{-(d-1)/2}$ is exactly the standard $L^1$--$L^\infty$ rate predicted by the fundamental solution (or alternatively stationary phase arguments). This is correlated with the extra spatial weights of the $K$ energy:
	$Q(K,\partial_t) |_{t = 0} = r^2 Q(\partial_t, \partial_t)$. 
\end{rmk}

\begin{rmk}[Kaluza-Klein backgrounds]\label{rmk:KK}
	That estimates for $\phi$ itself is available using \emph{only} the $\partial_t$ energy has important applications. Previously the only control for $\phi$ itself in the vector field method is via the $K$ energy as described in \cite{Klaine2001}. The availability of the $K$ energy however depends on the conformal symmetries of the wave equation: indeed, the vector field $K$ is also known as the ``conformally inverted time translation'' and can be obtained by pushing forward the vector field $\partial_t$ under the Lorentzian Kelvin transform. 
	Equations not exhibiting (an approximate version of) this conformal symmetry cannot be expected to have the same inequalities hold. 
	One place where the $K$ energy is not available is the case of Klein-Gordon equations. However, as was shown originally by Klainerman \cite{Klaine1985}, the $t^{-d/2}$ decay of $\phi$ itself is available in using the $\partial_t$ energy. 
	In our formulation, this observation is due to the fact that, for the Klein-Gordon equation 
	\[ \Box \phi - m^2 \phi = 0\]
	the analogue of the estimate \eqref{eq:tenergy2} takes the form
	\begin{multline*}
		\int_{\Sigma_\tau} \frac{1}{\tau^2 \cosh(\rho)} \sum_{i = 1}^d (L^i \phi)^2 + \frac{1}{\cosh(\rho)} (\partial_t\phi)^2 
		+ \cosh(\rho) m^2 \phi^2 ~\mathrm{dvol}_{\Sigma_\tau} \\
		\leq \int_{\{0\}\times \Real^d} (\partial_t\phi)^2 + \abs{\nabla \phi}^2 + m^2 \phi^2 \D{x}.
	\end{multline*}
	Then \thref{thm:GlobSob} implies that, after commuting with $L^\alpha$, 
	\[ \abs{\phi} \lesssim t^{-d/2} \]
	directly. These type or arguments were also used in \cite{LefMa2015} for handling coupled systems of wave and Klein-Gordon equations. 

	We note here that similar arguments can also be made for linear waves on Kaluza-Klein backgrounds. More precisely, consider the space-time $\Real^{1,d} \times S$ where $S$ is compact and equipped with some Riemannian metric $g_S$. Take the spacetime metric to be the product metric $g = \eta + g_S$. Consider a solution $\phi$ to the linear wave equation $\Box_g \phi = 0$ on this background. Due to the lack of conformal symmetry only the $\partial_t$-energy is available. For convenience denote by $\subD$ the derivatives tangent to $S$, then one can check that the analogue of \eqref{eq:tenergy2} in this case would be
	\begin{multline}
		\int_{\Sigma_\tau\times S}  \frac{1}{\tau^2 \cosh(\rho)} \sum_{i = 1}^d (L^i \phi)^2 + \frac{1}{\cosh(\rho)} (\partial_t\phi)^2 
		+ \cosh(\rho) \abs{\subD \phi}^2 ~\mathrm{dvol}_{\Sigma_\tau \times S} \\
		\leq \int_{\{0\} \times \Real^d \times S} (\partial_t\phi)^2 + \abs{\covD \phi}^2 + \abs{\subD \phi}^2 \D{x}\D{\sigma}.
	\end{multline}
	And the Global Sobolev inequalities then implies the following decay rates in the forward lightcone $\{t > r\}$:
	\begin{subequations}
		\begin{align}
			\abs{\phi} &\lesssim t^{1-d/2} \qquad (\text{only when }d \geq 3)\\
			\abs{L^i \phi} &\lesssim t^{1-d/2} \\
			\abs{\partial_t \phi} &\lesssim t^{1-d/2} (t+r)^{-1/2} (t-r)^{-1/2} \\
			\abs{\subD \phi} &\lesssim t^{-d/2} 
		\end{align}
	\end{subequations}
	which are exactly what one would expect from taking spectral projections on the $S$ component and treating $\phi$ as a sum of solutions to the wave equation and an infinite family of Klein-Gordon equations. 

	We emphasize that classical vector field method, integrating along the hyperplanes $\{t\}\times\Real^d \times S$, can only recover
	\[
		\abs{\partial\phi}, \abs{L^i \partial\phi}, \abs{\subD \partial \phi} \lesssim t^{(1-d)/2} (t-r)^{-1/2} 
	\]
	and in particular cannot see \emph{any} of the improved decay for $\subD \phi$ and its derivatives. 
\end{rmk}

We conclude this section with the decay estimates for $\phi$ when the dimension $d = 2$. For technical reasons our proof only works when $\phi$ has compactly supported initial data, and hence we only state the version available from the $K$ energy. 
\begin{prop}[$d = 2$, $K$-energy]
	Let $\phi$ solve the linear wave equation on $\Real^{1,2}$ with compactly supported initial data (e.g.\ satisfying the hypotheses of \thref{prop:ImprovedDecay}). Then
	\[ \abs{\phi} \lesssim \frac{\ln [(t+r)(t-r)]}{\sqrt{(t+r)(t-r)}}.\]
\end{prop}
\begin{proof}
	Without loss of generality we shall assume that $\phi$ satisfies the conditions of \thref{prop:ImprovedDecay}, this implies  the quantity in \eqref{eq:Kenergy} is bounded by the initial data. 

	Observe as the data, prescribed at time $t = 2$, is supported in the unit ball, this means by finite speed of propagation, on the support of $\phi$, when $t \geq 2$ it must also satisfy $t \geq r + 1$. Within the forward lightcone from the origin this inequality can be rewritten as
	\begin{equation}\label{eq:logloss}
		\tau(\cosh\rho - \sinh\rho) \geq 1 \iff \tau e^{-\rho} \geq 1 \iff \ln \tau \geq \rho. 
	\end{equation}
	Returning to \eqref{lem:hardy:2} we see that, when $\tau \geq 2$, 
	\begin{align*}
		\int_{\Sigma_\tau} \frac{1}{\cosh(\rho)} \phi^2 ~\mathrm{dvol}_{\Sigma_\tau} & \lesssim \int_{\Sigma_{\tau}} \frac{(1 + \rho^2)}{\cosh(\rho)} \sum (L^i \phi)^2 ~\mathrm{dvol}_{\Sigma_\tau} \\
		& \leq (1 + \ln\tau) \int_{\Sigma_\tau} \frac{1}{\cosh(\rho)} \sum (L^i \phi)^2 ~\mathrm{dvol}_{\Sigma_\tau} \\
		& \leq (1 + \ln\tau) \int_{\Sigma_\tau} \mathcal{K}_a (\partial_\tau)^a ~\mathrm{dvol}_{\Sigma_\tau}.
	\end{align*}
	Here, in the second inequality, we can use \eqref{eq:logloss} thanks to the compact-data assumption.

	Applying this estimate to $\phi$ and $L^\alpha \phi$ as before, we see by \thref{thm:GlobSob} that $\abs{\phi} \lesssim \tau^{-1} \ln(\tau)$ as claimed. 
\end{proof}

\section{Two-dimensional wave-maps} \label{sect:2dWM}

Let us now apply the method developed in the previous sections to study the global existence and decay of solutions to \eqref{eq:WMmain} under the assumption that its initial data has compact support. 
First we state the precise version of our theorem. 

\begin{thm}\label{thm:mainWM}
	Consider the future-evolution governed by \eqref{eq:WMmain} on $\Real^{1,2}$, with initial data prescribed at time $t = 2$ such that $\phi(2,x)$ and $\partial_t\phi(2,x)$ are supported in the ball of radius 1. Let $k \geq 4$ be a positive integer. Then there exists some $\epsilon > 0$ such that if 
		\[ \norm[H^{k+1}]{\phi(2, \text{---})} + \norm[H^{k}]{\partial_t\phi(2,\text{---})} < \epsilon,\]
		then there exists a future-global solution $\phi: [2,\infty)\times\Real^d \to \Real$. Furthermore, this solution obeys the following pointwise decay estimates:
		\begin{align*}
			\abs{\partial_t L^\alpha \phi} &\lesssim \frac{1}{\tau}, & & \abs{\alpha} \leq k-2;\\
			\abs{L^{\alpha} \phi} & \lesssim \frac{1}{\tau}, && 1 \leq \abs{\alpha} \leq k-1;\\
			\abs{\phi} & \lesssim \frac{\ln\tau}{\tau}. 
		\end{align*}
		Here $\tau = \sqrt{t^2 - \abs{x}^2}$ as defined previously.  (Note that by finite speed of propagation, within the support of $\phi$ we have the lower bound $\tau \geq \sqrt{2}$.)
\end{thm}

As the equation is semilinear, local existence theory using energy method is standard. In particular, for sufficiently small initial data it is clear that the solution exists at least up to, and including $\{\tau = 2\}$ (here we also use finite speed of propagation and boundedness of initial support). 
We concentrate on obtaining a priori energy bounds when $\tau \geq 2$. 
Define
\begin{gather}
	\mathcal{E}_\tau[\phi]^2 \eqdef \int_{\Sigma_\tau} \frac{1}{\tau^2 \cosh(\rho)} \sum_{i = 1}^2 (L^i\phi)^2 + \frac{1}{\cosh(\rho)} (\partial_t \phi)^2 ~\mathrm{dvol}_{\Sigma_\tau};\\
	\mathcal{F}_\tau[\phi]^2 \eqdef \int_{\Sigma_\tau} \frac{1}{\cosh(\rho)} \sum_{i = 1}^2 (L^i \phi)^2 + \frac{1}{\tau^2 \cosh(\rho)} [K\phi + t\phi]^2 ~\mathrm{dvol}_{\Sigma_\tau}.
\end{gather}
We have the following energy identities for $\tau_1 > \tau_0 > 2$
\begin{gather}
	\label{eq:EenergyEst} \mathcal{E}_{\tau_1}[\phi]^2  \leq \mathcal{E}_{\tau_0}[\phi]^2 + 2 \int_{\tau_0}^{\tau_1}  \abs{\Box \phi \cdot (\partial_t \phi)} ~\mathrm{dvol}_{\Sigma_{\tau}} \D{\tau}\\
	\label{eq:FenergyEst} \mathcal{F}_{\tau_1}[\phi]^2  \leq \mathcal{F}_{\tau_0}[\phi]^2 + 
	2\int_{\tau_0}^{\tau_1} \int_{\Sigma_\tau} \abs{\Box \phi \cdot (K\phi + t\phi)} ~\mathrm{dvol}_{\Sigma_{\tau}} \D{\tau}
\end{gather}
where we used that the space-time volume-element is exactly $\mathrm{dvol}_{\Sigma_{\tau}} \D{\tau}$ as seen in the metric decomposition 
\eqref{eq:metricNew}.

The basic strategy is standard: we wish to estimate $\mathcal{E}[L^\alpha \phi]$ and $\mathcal{F}[L^\alpha \phi]$ for all $\abs{\alpha}$ less than some fixed constant. 
This will require estimating integrals of the forms
\begin{subequations}
	\begin{equation}\label{eq:bulkintt}
	\int_{\Sigma_{\tau}} L^{\alpha_1} \phi \cdot \eta( \D* L^{\alpha_2} \phi, \D* L^{\alpha_3} \phi)\cdot \partial_t L^{\alpha}\phi ~\mathrm{dvol}_{\Sigma_\tau} 
\end{equation}
and
	\begin{equation}\label{eq:bulkintK}
	\int_{\Sigma_{\tau}} L^{\alpha_1} \phi \cdot\eta( \D*L^{\alpha_2} \phi, \D* L^{\alpha_3} \phi) \cdot (K+t)L^{\alpha}\phi ~\mathrm{dvol}_{\Sigma_\tau} 
\end{equation}
\end{subequations}
for $\abs{\alpha_1} + \abs{\alpha_2} + \abs{\alpha_3} = \abs{\alpha}$. In writing down the integrals above we implicitly used that vector fields acts on scalars by Lie differentiation, and $\eta$ is invariant under Lorentz boosts $L^i$, and that exterior differentiation commutes with Lie differentiation, so that if $\phi$ solves \eqref{eq:WMmain}, then $L^\alpha\phi$ solves an equation of the form 
\begin{equation}
	\Box L^\alpha \phi = \sum c_{\alpha_1, \alpha_2, \alpha_3; \alpha} L^{\alpha_1} \phi \eta(\D* L^{\alpha_2}\phi, \D* L^{\alpha_3}\phi)
\end{equation}
where $\abs{\alpha_1} + \abs{\alpha_2} + \abs{\alpha_3} = \abs{\alpha}$, and $c_{\alpha_1, \alpha_2, \alpha_3; \alpha}$ are combinatorial constants. 

Therefore we are led to consider integrals of the form 
\[ \int_{\Sigma_\tau} \zeta \eta(\D*\psi, \D*\phi) \partial_t \xi ~\mathrm{dvol}_{\Sigma_{\tau}} ~~ \text{ and } ~~ \int_{\Sigma_\tau} \zeta \eta(\D*\psi, \D*\phi) (K\xi + t \xi) ~\mathrm{dvol}_{\Sigma_{\tau}}\]
where the functions $\zeta, \psi, \phi, \xi$ stand in place of $L^\alpha$ derivatives of the original unknown. 
To get good control, we begin by decomposing the form
\[ \eta(\D*\psi, \D*\phi) = (h_\tau)^{-1}(\D*\psi, \D*\phi) - (\partial_{\tau} \psi)(\partial_{\tau}\phi). \]
Using that
\begin{align*}
	\partial_\tau &= \frac{1}{t\tau} K - \frac{r}{t\tau} \partial_\rho \\
		& = \frac{\tau}{t} \partial_t + \frac{r}{t\tau} \partial_\rho
\end{align*}
we have
\[ \eta(\D*\psi, \D*\phi) = (h_\tau)^{-1}(\D*\psi,\D*\phi) - \big[ \frac{1}{t\tau} K\psi - \frac{r}{t\tau} \partial_\rho \psi \big] \big[ \frac{\tau}{t} \partial_t \phi + \frac{r}{t\tau} \partial_\rho \phi \big]. \]
Introducing the short hand $\abs{L\phi} \eqdef \big[ \sum_{i = 1}^2 (L^i \phi)^2 \big]^{\frac12}$ we can bound
\begin{multline*}
	\abs{\eta(\D*\psi, \D*\phi)} \lesssim \frac{1}{\tau^2} \abs{L\psi} \abs{L\phi}\\
		+\frac{1}{t^2} \abs{ K\psi \partial_t\phi} + \frac{r}{t^2} \abs{\partial_\rho \psi \partial_t \phi} 
		+ \frac{r}{t^2 \tau^2} \abs{K\psi \partial_\rho \phi} + \frac{r^2}{t^2 \tau^2} (\partial_\rho \phi) (\partial_\rho \psi).
\end{multline*}
This we can bound by 
\begin{multline}\label{eq:bilinearest}
	\abs{\eta(\D*\psi, \D*\phi)} \lesssim \frac{1}{\tau^2} \abs{L\psi} \abs{L\phi}
		+ \frac{1}{t} \abs{\partial_t\phi} \abs{L\psi} \\
		+\frac{1}{t^2} \abs{ (K+t)\psi} \abs{\partial_t\phi} + \frac1{t} \abs{\psi} \abs{\partial_t \phi}
		+ \frac{1}{t \tau^2} \abs{(K+ t)\psi}\abs{L \phi} + \frac{1}{\tau^2} \abs{\psi} \abs{L\phi}
\end{multline}
where we simplified using $r / t \leq 1$.

\begin{rmk}
	We use \eqref{eq:bilinearest} rather than the more common decomposition 
	\[ \abs{\eta(\D*\psi,\D*\phi)} \lesssim \abs{\partial_t\psi}\abs{\partial_t\phi} + \frac{1}{t^2} \abs{L\psi}\abs{L\phi} \]
	to take full advantage of both the $L^2$ and $L^\infty$ controls we have on $(K + t)\psi$, which gives us additional decay.  
\end{rmk}

With this estimate, we can first control the nonlinearity when the quantity $\alpha_1$ in \eqref{eq:bulkintt} and \eqref{eq:bulkintK} has order $\abs{\alpha_1} \leq \abs{\alpha} - 1$.

\begin{lem}[Estimates for $\abs{\alpha_1} \leq \abs{\alpha} - 1$]\label{lem:nonborder}
	We have the estimates (the $L^\infty$ norms are taken along $\Sigma_\tau$)
	\begin{gather*}
		\int_{\Sigma_\tau} \zeta \eta(\D*\psi, \D*\phi) \partial_t \xi ~\mathrm{dvol}_{\Sigma_{\tau}} \lesssim 
			\frac{\ln \tau}{\tau} \norm[L^\infty]{\zeta} \big[ \norm[L^\infty]{L\phi} + \norm[L^\infty]{\partial_t\phi} \big]  \mathcal{F}_\tau[\psi]\mathcal{E}_\tau[\xi], \\
		\int_{\Sigma_\tau} \zeta \eta(\D*\psi, \D*\phi) (K\xi + t \xi) ~\mathrm{dvol}_{\Sigma_{\tau}} \lesssim 
			\ln\tau \norm[L^\infty]{\zeta} \big[\norm[L^\infty]{L\phi} + \norm[L^\infty]{\partial_t\phi} \big] \mathcal{F}_\tau[\psi] \mathcal{F}_\tau[\xi].
	\end{gather*}
\end{lem}
\begin{proof}
	The proof of the two cases are similar, we focus on the harder case which is the second inequality. The basic idea is to put all the terms involving $\xi$ and $\psi$ in weighted $L^2$ (controlled by $\mathcal{E}[\xi], \mathcal{F}[\xi], \mathcal{F}[\psi]$), and the remainder ($\zeta$ and $\phi$) in $L^\infty$. 
	At parts of the argument we will also use the fact that by our finite speed of propagation property, $\cosh(\rho) \leq e^\rho \leq \tau$ and hence $\frac{t}{\tau^2} \leq 1$.
	
	Observe that the $\mathcal{F}[\xi]^2$ controls the square integral of $(K \xi + t \xi) / \sqrt{t \tau}$. So it suffices to show that $\sqrt{t\tau} \eta(\D*\psi, \D*\phi)$ is square integrable on $\Sigma_\tau$. Observe now that $\mathcal{F}[\psi]^2$ controls the square integral of $\abs{L\psi} / \sqrt{\cosh(\rho)}$, $(K\psi + t\psi)/ \sqrt{t\tau}$ and $\abs{\psi}/[\ln(\tau) \sqrt{\cosh(\rho)}]$ (the last through \thref{lem:hardy:2} and finite speed of propagation). 
	So we can check each term that appears on the right of \eqref{eq:bilinearest}:
	\begin{align*}
		\frac{\sqrt{t\tau}}{\tau^2} \abs{L\psi} \abs{L\phi} & = \frac{\cosh\rho}{\tau} \abs{L\phi} \cdot \frac{1}{\sqrt{\cosh\rho}} \abs{L\psi} \\
		\frac{\sqrt{t\tau}}{t} \abs{L\psi} \abs{\partial_t\phi} & = \abs{\partial_t\phi} \cdot \frac{1}{\sqrt{\cosh\rho}} \abs{L\psi}\\
		\frac{\sqrt{t\tau}}{t^2} \abs{(K+t)\psi} \abs{\partial_t \phi} & =  \frac{\tau}{t} \abs{\partial_t\phi} \cdot \frac{1}{\sqrt{t\tau}} \abs{(K+t)\psi}\\
		\frac{\sqrt{t\tau}}{t} \abs{\psi} \abs{\partial_t \phi} & =  \ln(\tau) \abs{\partial_t\phi}  \cdot \frac{1}{\ln(\tau) \sqrt{\cosh\rho}} \abs{\psi}\\
		\frac{\sqrt{t\tau}}{t\tau^2} \abs{(K+t)\psi} \abs{L\phi} & =  \frac{1}{\tau} \abs{L\phi} \cdot \frac{1}{\sqrt{t\tau}} \abs{(K+t)\psi}\\
		\frac{\sqrt{t\tau}}{\tau^2} \abs{\psi} \abs{L\phi} & = \frac{t \ln \tau}{\tau^2} \abs{L\phi} \cdot \frac{1}{\ln(\tau) \sqrt{\cosh\rho}} \abs{\psi}
	\end{align*}
	Using that $t/\tau^2 \leq 1$ by finite speed of propagation we see our claim is proved. 
\end{proof}

For taking care of the terms where $\abs{\alpha_1} = \abs{\alpha}$, we will place the $\zeta$ term in $L^2$, and $(K + t)\psi$ in $L^\infty$. 
The corresponding integral estimates are
\begin{lem}[Estimates for $|\alpha_1| = |\alpha|$]\label{lem:borderline}
	We have the estimates (the $L^\infty$ norms are taken along $\Sigma_\tau$)
	\begin{align*}
		\int_{\Sigma_\tau} & (L\zeta) \eta(\D*\psi, \D*\phi) \partial_t \xi ~ \mathrm{dvol}_{\Sigma_{\tau}} \lesssim \\
										  & \qquad \frac1\tau \big[ \frac1t \norm[L^\infty]{K \psi + t\psi} + \norm[L^\infty]{L\psi} + \norm[L^\infty]{\psi} \big] \cdot \big[ \norm[L^\infty]{\partial_t\phi} + \norm[L^\infty]{L\phi} \big] \mathcal{F}[\zeta] \mathcal{E}[\xi],\\
		\int_{\Sigma_\tau} & (L\zeta) \eta(\D*\psi, \D*\phi) (K\xi + t \xi) ~ \mathrm{dvol}_{\Sigma_{\tau}} \lesssim \\
			& \qquad \big[ \frac1t \norm[L^\infty]{K \psi + t\psi} + \norm[L^\infty]{L\psi} + \norm[L^\infty]{\psi} \big] \cdot \big[ \norm[L^\infty]{\partial_t\phi} + \norm[L^\infty]{L\phi} \big] \big] \mathcal{F}[\zeta] \mathcal{F}[\xi].
	\end{align*}
\end{lem}
\begin{proof}
	Again we focus on the second inequality, the first is similar. In this situation we must put $L\zeta$ in $L^2$, and both $\phi$ and $\psi$ in $L^\infty$. The square integral of $(L\zeta) / \sqrt{\cosh\rho}$ is controlled by $\mathcal{F}_\tau[\zeta]^2$, so it suffices to compute the $L^\infty$ estimates of $t \eta(\D*\psi, \D*\phi)$. The results follow straightforwardly from \eqref{eq:bilinearest}, after using again that $\cosh(\rho) \leq \tau$ from finite speed of propagation. 
\end{proof}
 
Next let us introduce the following shorthand notations:
\begin{gather*}
	\mathcal{E}_{\tau_0, \tau_1, k} = \sup_{\tau\in [\tau_0,\tau_1]} \sum_{\alpha \in \mathfrak{I}^k} \mathcal{E}_\tau[L^\alpha \phi] \\
	\mathcal{F}_{\tau_0, \tau_1, k} = \sup_{\tau\in [\tau_0,\tau_1]} \sum_{\alpha \in \mathfrak{I}^k} \mathcal{F}_\tau[L^\alpha \phi] \\
	\mathcal{E}_{\tau_0, \tau_1, \leq k} = \sum_{j = 0}^k \mathcal{E}_{\tau_0, \tau_1, j} \\
	\mathcal{F}_{\tau_0, \tau_1, \leq k} = \sum_{j = 0}^k \mathcal{F}_{\tau_0, \tau_1, j}
\end{gather*}
The global Sobolev inequality \thref{thm:GlobSob} states, when $d = 2$, that, for $\tau \in [\tau_0, \tau_1]$ and $\psi = L^\alpha \phi$ with $\abs{\alpha} = k$:
\begin{subequations}
\begin{gather}
	\label{eq:GStp} \abs{\partial_t \psi} \lesssim \tau^{-1} \mathcal{E}_{\tau_0, \tau_1, \leq k+2} \\
	\label{eq:GSlp}\abs{L\psi} \lesssim \tau^{-1} \mathcal{F}_{\tau_0, \tau_1, \leq k+2} \\
	\label{eq:GSkp}\abs{K\psi + t \psi} \lesssim \mathcal{F}_{\tau_0, \tau_1, \leq k + 2} \\
	\label{eq:GSp}\abs{\psi} \lesssim \ln(\tau) \tau^{-1} \mathcal{F}_{\tau_0, \tau_1, \leq k+1}
\end{gather}
\end{subequations}
where in the last estimate \eqref{eq:GSp} we used the compact support assumption,  and in the second to last estimate \eqref{eq:GSkp} we used the commutator estimate \eqref{eq:Kcommu}.
 
Applying the global Sobolev inequality results above to \thref{lem:nonborder} and \thref{lem:borderline}, we obtain (here interpreting $(\zeta,\psi,\varphi,\xi) = (L^{\alpha}\phi, L^{\beta}\phi, L^{\gamma}\phi, L^\delta\phi)$)
\begin{align*}
	\int_{\Sigma_\tau} \zeta \eta(\D*\psi, \D*\varphi) \partial_t \xi & ~\mathrm{dvol}_{\Sigma_{\tau}} \lesssim \\
									  &	\frac{(\ln \tau)^2}{\tau^3} \mathcal{F}_{\tau_0,\tau,\leq\abs{\alpha}+1}\big[ \mathcal{E}_{\tau_0,\tau,\leq\abs{\gamma}+2} + \mathcal{F}_{\tau_0,\tau,\leq\abs{\gamma}+2} \big]  \mathcal{F}_{\tau_0,\tau,\abs{\beta}} \mathcal{E}_{\tau_0,\tau,\abs{\delta}} \\
	\int_{\Sigma_\tau} \zeta \eta(\D*\psi, \D*\varphi) (K\xi + t \xi)& ~\mathrm{dvol}_{\Sigma_{\tau}} \lesssim \\
									 &\frac{(\ln \tau)^2}{\tau^2} \mathcal{F}_{\tau_0,\tau,\leq\abs{\alpha}+1}\big[ \mathcal{E}_{\tau_0,\tau,\leq\abs{\gamma}+2} + \mathcal{F}_{\tau_0,\tau,\leq\abs{\gamma}+2} \big] 
		\mathcal{F}_{\tau_0, \tau, \abs{\beta}} \mathcal{F}_{\tau_0, \tau, \abs{\delta}}, \\
	\int_{\Sigma_\tau} (L\zeta) \eta(\D*\psi, \D*\varphi) \partial_t \xi & ~ \mathrm{dvol}_{\Sigma_{\tau}} \lesssim \\
									     &	\frac{\ln(\tau)}{\tau^3} \mathcal{F}_{\tau_0,\tau,\leq \abs{\beta}+2}  \big[\mathcal{E}_{\tau_0,\tau,\leq\abs{\gamma}+2} + \mathcal{F}_{\tau_0,\tau,\leq\abs{\gamma}+2}   \big] \mathcal{F}_{\tau_0, \tau, \abs{\alpha}} \mathcal{E}_{\tau_0, \tau, \abs{\delta}},\\
	\int_{\Sigma_\tau}  (L\zeta) \eta(\D*\psi, \D*\varphi) (K\xi + t \xi) & ~ \mathrm{dvol}_{\Sigma_{\tau}} \lesssim \\
									      &\frac{\ln(\tau)}{\tau^2} \mathcal{F}_{\tau_0,\tau,\leq \abs{\beta}+2} \big[\mathcal{E}_{\tau_0,\tau,\leq\abs{\gamma}+2} + \mathcal{F}_{\tau_0,\tau,\leq\abs{\gamma}+2}   \big]     \mathcal{F}_{\tau_0, \tau, \abs{\alpha}} \mathcal{F}_{\tau_0, \tau,\abs{\delta}}.
\end{align*}
Combining with  \eqref{eq:EenergyEst} we arrive at 
\begin{multline}\label{eq:EEE}
	\mathcal{E}_{\tau_0, \tau_1, \leq k} - \mathcal{E}_{\tau_0, \tau_0, \leq k} \lesssim \\
		\int_{\tau_0}^{\tau_1} \frac{(\ln \tau)^2}{\tau^{3}} \big( \mathcal{F}_{\tau_0, \tau, \leq k/2+2} + \mathcal{E}_{\tau_0, \tau, \leq k/2+2}\big) \mathcal{F}_{\tau_0, \tau, \leq k/2+2} \mathcal{F}_{\tau_0, \tau, \leq k} \D{\tau}.
\end{multline}
Similarly we get for the $K$-energy
\begin{multline}\label{eq:EEF}
	\mathcal{F}_{\tau_0, \tau_1, \leq k} - \mathcal{F}_{\tau_0, \tau_0, \leq k} \lesssim \\
		\int_{\tau_0}^{\tau_1} \frac{(\ln \tau)^2}{\tau^{2}} \big( \mathcal{F}_{\tau_0, \tau, \leq k/2+2} + \mathcal{E}_{\tau_0, \tau, \leq k/2+2}\big) \mathcal{F}_{\tau_0, \tau, \leq k/2+2} \mathcal{F}_{\tau_0, \tau, \leq k} \D{\tau}.
\end{multline}

The above estimates imply the following bootstrapping estimate:
\begin{prop}[Bootstrap]
	Fix $k \geq 4$ to be an integer. There exists a universal constant $C$ depending on $k$ such that the following holds:
	Let $\phi$ solve \eqref{eq:WMmain} on the spacetime region sandwiched between $\Sigma_{\tau_0}$ and $\Sigma_{\tau_1}$ with $2 \leq \tau_0 < \tau_1$ be such that $\phi$ vanishes when $e^\rho \geq \tau$. Suppose $\phi$ satisfies the following assumptions:
	we have the initial data bound
	\begin{equation}\label{eq:idbound}
		\mathcal{E}_{\tau_0, \tau_0,\leq k} + \mathcal{F}_{\tau_0,\tau_0,\leq k} \leq \epsilon_0;
	\end{equation}
	and the bootstrap bounds
	\begin{equation}\label{eq:BAbound}
		\mathcal{E}_{\tau_0, \tau_1,\leq k} + \mathcal{F}_{\tau_0, \tau_1, \leq k} \leq \delta_0.
	\end{equation}
	Then we have the improved bounds
	\begin{equation}
		\mathcal{E}_{\tau_0, \tau_1, \leq k} + \mathcal{F}_{\tau_0, \tau_1, \leq k} \leq \epsilon_0 + C \delta_0^3.
	\end{equation}
\end{prop}
\begin{proof}
	The proposition follows immediately from the integrability of the function $\frac{(\ln \tau)^2}{\tau^2}$ on $[\tau_0, \infty)$. 	
\end{proof}

In particular, if $\delta_0$ is such that $C \delta_0^2 \leq \frac13$, and $\epsilon_0 \leq \frac1{3} (\ln2) \delta_0$, then the conclusion of the above proposition implies
	\begin{equation}
		\mathcal{E}_{\tau_0, \tau_1, \leq k} + \mathcal{F}_{\tau_0, \tau_1, \leq k} \leq \frac23 \delta_0.
	\end{equation}
In this case, the (semi-)global existence part of our main \thref{thm:mainWM} follows by a continuity argument, and the decay estimates follow from an application of \thref{thm:GlobSob} to the energy bounds derived above. 

\begin{rmk}\label{rmk:otherder}
	In the arguments above we were able to close the argument by commuting the equation with \emph{only} Lorentz boost vector fields $L^i$. 
	If we were to wish to also control the derivatives relative to the full algebra of the symmetries of Minkowski space, generated by the vector fields 
	\begin{equation}\label{eq:fullcommutators}
		\mathfrak{Z} \eqdef \{ \partial_t, \partial_{x^i}, \Omega_{ij}, L^i\},
	\end{equation}
	we see that the argument goes through largely unchanged, with the only significant modification coming in \thref{lem:borderline}. As in this case we can no longer guarantee that the derivative on $\zeta$ is an $L$ derivative, we need to use the following observations:
	\begin{enumerate}
		\item In the case of $\partial_t \zeta$, whereas in the proof of \thref{lem:borderline} we used that the square integral of $(L\zeta) / \sqrt{\cosh\rho}$ is controlled by $\mathcal{F}_\tau[\zeta]^2$, we have that the square integral of $\partial_t \zeta / \sqrt{\cosh\rho}$ is controlled by $\mathcal{E}_\tau[\zeta]^2$, and hence we replace the $\mathcal{F}[\zeta]$ on the right of the inequality by $\mathcal{E}[\zeta]$. 
		\item In the case of $\Omega_{ij} \zeta$, the decomposition \eqref{eq:omegadecomp} and the boundedness of $\frac{x}{t}$ allows us to bound the square integral of $(\Omega_{ij}\zeta) / \sqrt{\cosh\rho}$ also by $\mathcal{F}[\zeta]^2$. 
		\item Finally, in the case of the $\partial_{x^i}$, we decompose
			\[ \partial_{x^i} = \frac{1}{t} L^i - \frac{x^i}{t} \partial_t \]
			and therefore the square integral of $\partial_{x^i} \zeta / \sqrt{\cosh\rho}$ is also bounded by $\mathcal{E}[\zeta]^2$. 
	\end{enumerate}
	These changes can be accommodated in \eqref{eq:EEE} and \eqref{eq:EEF} by changing the factor $\mathcal{F}_{\tau_0,\tau, \leq k}$ to $(\mathcal{F}_{\tau_0,\tau, \leq k} + \mathcal{E}_{\tau_0, \tau, \leq k})$, following which the bootstrap argument proceeds unchanged. 
\end{rmk}

\section{A perturbation result}\label{sect:largeDpert}

The argument in the previous section immediately implies a ``stability of large-data dispersive solutions'' result, in the spirit of and generalizing the results of \cite{Sideri1989}.
For convenience, fix throughout the section a positive integer $k \geq 4$. Implicit constants are understood to depend on $k$. 

First, let us make precise what we mean by ``dispersive solutions''. 
Observe that the global solutions of \thref{thm:mainWM} are automatically ``dispersive in the forward light-cone'' per the following definition. 

\begin{defn}\label{defn:dispsoln}
	Let $\Phi$ be a \emph{global} solution of \eqref{eq:WMmain}. We say that it is \emph{dispersive in the forward light-cone} if there exists a constant $M > 0$ such that the following bounds hold on the set $\{ t > \max(1 + \abs{x},2)\}$:
	\begin{itemize}
		\item Uniform-in-$\tau$  $L^\infty$ bounds for all $\alpha \in \mathfrak{I}^{\leq k-2}$. 
			\[
				\begin{cases}
					\abs{\partial_t L^\alpha \Phi} \leq M \tau^{-1} \\
					\abs{L^i L^\alpha \Phi} \leq M \tau^{-1}  \\
					\abs{(K + t) L^\alpha \phi} \leq M  \\
					\abs{\Phi} \leq M \ln(\tau) \tau^{-1} 
				\end{cases}
			\]
		\item Space-time $L^2$ bounds. For every $\alpha \in \mathfrak{I}^{\leq k}$,
			\[
				\int_{\{t > \max(1+ \abs{x},2)\}} \frac{ \ln(\tau)^2}{\tau^2\cosh(\rho)} \left[ (\partial_t L^\alpha \Phi)^2 + (L L^\alpha\Phi)^2 + \frac{1}{\tau^2} ([K + t] L^\alpha \Phi)^2 \right] ~\mathrm{dvol} \leq M^2
			\]
	\end{itemize}
\end{defn}

\begin{rmk}
	The background solutions considered in \cite{Sideri1989}, which are formed from looking at solutions to the wave-maps equation whose range restrict to a geodesic on the target manifold, are automatically conjugate to a solution of the linear wave equation and therefore also satisfy the above bounds making them dispersive in the forward light-cone. 
\end{rmk}

We will study here perturbations of a given solution $\Phi$. That is, we look for $\phi$ such that $\Phi + \phi$ is a solution to \eqref{eq:WMmain} and study the corresponding initial problem. 
The equation satisfied by $\phi$ is 
\begin{equation}\label{eq:WMpert}
	\Box \phi = (\Phi + \phi) \cdot \eta(\D*(\Phi + \phi), \D*(\Phi + \phi)) - \Phi \cdot \eta(\D\Phi,\D\Phi). 
\end{equation}

\begin{thm}\label{thm:pert}
	Let $\Phi$ be a \emph{global} solution of \eqref{eq:WMmain} that is dispersive in the forward light-cone, with the associated constant $M$. 
	Consider the future-evolution governed by \eqref{eq:WMpert} on $\Real^{1,d}$, with initial data prescribed at time $t = 2$ such that $\phi(2,x)$ and $\partial_t\phi(2,x)$ are supported in the ball of radius $1$. Then there exists some constant $\lambda > 0$ \emph{independent of $M$} such that if 
	\[ \norm[H^{k+1}]{\phi(2,\text{---})} + \norm[H^{k}]{\partial_t\phi(2,\text{---})} \leq \frac{1}{\exp(\lambda (1+M)^2)} \]
	then there exists a future-global solution $\phi:[2,\infty)\times\Real^d \to \Real$. 
\end{thm}

\begin{rmk}
	We keep track of the $M$ dependence for applications in the next section; for stability of large-data dispersive solutions the precise dependence is not needed. 
	The large exponential loss in $M$ is due entirely to the \emph{linear}-in-$\phi$ terms that appear on the right of \eqref{eq:WMpert}, and seems hard to avoid. 
	In the next section this exponential loss implies that our method only addresses very strongly localized non-compact initial data, with tails that decay super-exponentially. 
\end{rmk}

\begin{proof}[Sketch]
	We proceed largely as in the previous section, and apply the estimates from \thref{lem:nonborder} and \thref{lem:borderline}, with the slots $(\zeta, \psi, \phi)$ in the Lemmas replaced by $L^\alpha \phi$ and $L^\beta \Phi$. 
	If we make the bootstrap assumptions \eqref{eq:idbound} and \eqref{eq:BAbound} as before, then in the place of \eqref{eq:EEE} and \eqref{eq:EEF} we obtain the energy estimate
	\begin{multline}
		\mathcal{E}_{\tau_0,\tau_1,\leq k} + \mathcal{F}_{\tau_0,\tau_1,\leq k} - \epsilon_0 \lesssim \int_{\tau_0}^{\tau_1} \frac{(\ln \tau)^2}{\tau^2} (M + \delta_0)^2 (\mathcal{E}_{\tau_0,\tau,\leq k} + \mathcal{F}_{\tau_0,\tau, \leq k}) \D\tau \\
		+ \sum_{\alpha \in \mathfrak{I}^{\leq k}}\int_{\tau_0}^{\tau_1} \frac{(\ln\tau)^2}{\tau^2} (M + \delta_0) (\mathcal{E}_{\tau_0, \tau, \leq k/2 + 2} + \mathcal{F}_{\tau_0, \tau, \leq k/2 + 2}) \cdot \\
		\left(\int_{\Sigma_\tau} \abs{\partial_tL^\alpha \Phi}^2 + \abs{L L^\alpha\Phi}^2 + \frac{1}{\tau^2} \abs{[K + t]L^\alpha \Phi})^2 \frac{1}{\cosh\rho} ~\mathrm{dvol}_{\Sigma_\tau}\right)^{1/2} \D\tau.
	\end{multline}
	Applying Gr\"onwall's inequality we get (where $\lambda_0$ is the implicit structural constant that appears in the previous inequality):
	\begin{multline}
		\mathcal{E}_{\tau_0,\tau_1,\leq k} + \mathcal{F}_{\tau_0,\tau_1,\leq k} \leq \epsilon_0 \exp\Bigg( \lambda_0 (M + \delta_0)^2 \int_{\tau_0}^{\tau_1} \frac{(\ln \tau)^2}{\tau^2} \D\tau + \lambda_0 (M + \delta_0) \int_{\tau_0}^{\tau_1}   \\
		\sum_{\alpha\in \mathfrak{I}^{\leq k}} \Big( \int_{\Sigma_\tau} \abs{\partial_tL^\alpha \Phi}^2 + \abs{L L^\alpha \Phi}^2 + \frac{1}{\tau^2} \abs{[K + t]L^\alpha \Phi})^2 \frac{1}{\cosh\rho} ~\mathrm{dvol}_{\Sigma_\tau} \Big)^{1/2} \D\tau \Bigg).
	\end{multline}
	The first integral inside the exponential is obviously bounded. The second integral is also bounded after noting that $\tau^{-2} \ln(\tau)^2 \D\tau$ is a finite measure and so we can apply H\"older's inequality to control $L^1$ by $L^2$, whereupon our space-time $L^2$ bound comes into play. If $\delta_0$ is initially chosen to be $1$, then we can find $\lambda > 0$ such that the above inequality simplifies to 
	\begin{equation}
		\mathcal{E}_{\tau_0, \tau_1, \leq k} + \mathcal{F}_{\tau_0, \tau_1, \leq k} \leq \epsilon_0 \exp(\frac12 \lambda(1 + M)^2).
	\end{equation}
	So provided we choose $\epsilon_0 < \exp(-\lambda(1 + M)^2)$ we obtain the \emph{improved} bootstrap bound
	\[ \mathcal{E}_{\tau_0, \tau_1, \leq k} + \mathcal{F}_{\tau_0, \tau_1, \leq k} \leq  \exp(-\frac12 \lambda) < 1 = \delta_0 \]
	from which global existence follows from a standard continuity argument. 
\end{proof}

\begin{rmk}
	The phenomenon discussed here holds rather generally for large-classes of semilinear equations. In particular, one can regard the above as related to the construction described in \cite{WanYu2016} for semilinear wave equations. There a version of Christodoulou's \emph{short pulse} method is used to construct large data solutions to such equations. Their method can be interpreted as constructing first a future-global, dispersive approximate solution via the almost-radial initial data (which can, heuristically speaking, by studying an approximate evolution for the radial part), and perturbing it using that both the approximate solution and residual error terms have good dispersive decay.  Our approach provides in a general setting, the argument for the perturbative step.  
\end{rmk}

\begin{rmk}\label{rmk:otherderag}
	The perturbed solution $\Phi + \phi$ also satisfies uniform-in-$\tau$ $L^\infty$ bounds and space-time $L^2$ bounds by triangle inequality. Thus for $\lambda > 0$ sufficiently large the same bounds as $\Phi$ hold, but with $M$ replaced by $(M+1)$. 
\end{rmk}

\begin{rmk}
	The \thref{defn:dispsoln} and \thref{thm:pert} are stated using only Lorentz boost commutators. However, in view of \thref{rmk:otherder}, the same results hold if we use also the full set $\mathfrak{Z}$ of \eqref{eq:fullcommutators}. 
	For the result in the next section, having available the full set of commutators is convenient (but not necessary). 
\end{rmk}

\section{Extension to non-compact data} \label{sect:noncompact}

The discussion in the previous section only deals with compactly-supported initial data, which is a common problem with studying nonlinear hyperbolic equations using the hyperboloidal-based energy methods. 
In \cite{LeFMa2017}, LeFloch and Ma introduced the so-called Euclidean-Hyperboloidal method to overcome this difficulty in the context of the stability problem of Minkowski space. 
Their method requires gluing asymptotically Euclidean slicings to the hyperboloids and is not obviously compatible with the weight loss of Lemma \ref{lem:hardy:2}. 
For the wave-map problem we circumvent this difficulty by taking advantage of the \emph{scaling invariance} of the wave-map equation in $d = 2$, this allows us to show that in the case of \emph{strongly localized} (but not necessarily of compact support) initial data, a small-data global existence theorem persists. 

We begin by observing that if $\phi$ is a solution to \eqref{eq:WMmain}, then so is 
\begin{equation}
	S_m \phi(t,x) \eqdef \phi(2^m (t-2) + 2, 2^m x), 
\end{equation}
the extra time-translation is to keep the constant time slice $\{t = 2\}$ invariant under the transformation. 

Unfortunately, as the argument requires studying the initial data in $H^{k}$ norm (and not a scale invariant norm), the scaling argument does not automatically allow us to approximate non-compact initial data by compact ones. Specifically, observe that for any multi-index $\alpha$
\[ \int_{\{2\} \times \Real^d} |\partial^\alpha S_m\phi(t,x)|^2 \D x = 2^{2m (|\alpha| - 1)} \int_{\{2\}\times \Real^d} |\partial^\alpha \phi(t,x)|^2 \D x. \]
This scaling property rules out na\"\i{}vely approaching non-compact initial data via a cut-off/rescaling procedure, even if one were to measure the initial data in a weighted Sobolev space (since to control the higher-order initial energy, we must measure the initial data against non-singular weights). 

Our approach is to perform a spatial dyadic decomposition of the initial data and apply the Perturbation Theorem \ref{thm:pert} to iteratively construct the solution. 
More precisely, let $\chi_0$ be a smooth, positive function of compact support on $\Real^2$ such that it vanishes outside the ball of radius 1 and is $\equiv 1$ on the ball of radius $1/2$. 
For $m \geq 1$ we can set
\begin{equation}
	\chi_m(x) = \chi_0( 2^{-m} x) - \chi_0( 2^{1-m} x)
\end{equation}
so that $\chi_m$ is supported in the annulus $\abs{x} \in [2^{m-2}, 2^m]$. 

Consider now the initial value problem for \eqref{eq:WMmain}, with initial data prescribed at $t = 2$
\begin{equation}
	\begin{cases}
		\phi(2,x) = \mathring{\phi}(x) \\
		\partial_t\phi(2,x) = \mathring{\varphi}(x) 
	\end{cases}.
\end{equation}
We decompose the initial data spatially. For $m \geq 0$, 
\begin{gather}
	\mathring{\phi}_m \eqdef \chi_m \cdot \mathring{\phi},\\
	\mathring{\varphi}_m \eqdef \chi_m \cdot \mathring{\varphi}.
\end{gather}
Let $\phi_0$ solve \eqref{eq:WMmain} with initial data $(\mathring{\phi}_0, \mathring{\varphi}_0)$.
For $m \geq 1$ define iteratively first $\Phi_m = \sum_{n = 0}^{m-1} \phi_n$, and $\phi_m$ the solution to
\begin{equation}
	\Box \phi_m = (\Phi_m + \phi_m) \eta(\D*(\Phi_m + \phi_{m}) , \D*(\Phi_m + \phi_m)) -\Phi_{m} \eta(\D*\Phi_{m}, \D* \Phi_m)
\end{equation}
with initial data $(\mathring{\phi}_m, \mathring{\psi}_m)$. 
Then provided we can show that each $\phi_m$ exists globally, the infinite sum $\phi = \sum_{m = 0}^\infty \phi_m$ will be a solution to the original problem. The convergence of this sum is automatically \emph{uniform} on any compact space-time subset: in fact, due to finite-speed-of-propagation, $\phi_m$ vanishes within the set $\{ \abs{t} + \abs{x} \leq 2^{m-2} \}$, and hence for every $(t,x)$ the sum $\phi$ converges after finitely many terms. 

In order to apply the uniform estimates of \thref{thm:pert}, we will rescale the problem solved by $\phi_m$. Define the scaling operator
\begin{equation}
	\mathring{S}_m f(x) \eqdef f(2^m x)
\end{equation}
Then we have that $S_m \phi_m$ solves the equation
\begin{multline}
	\Box S_m\phi_m = (S_m\Phi_m + S_m\phi_m) \eta(\D*(S_m\Phi_m + S_m\phi_{m}) , \D*(S_m\Phi_m + S_m\phi_m))\\ -(S_m\Phi_{m}) \eta(\D*S_m\Phi_{m}, \D* S_m\Phi_m)
\end{multline}
with initial data $(\mathring{S}_m \mathring{\phi}_m, 2^{m} \mathring{S}_m \mathring{\varphi}_m)$.

(Notice that our scaling operators form a semi-group under composition: $S_m = (S_1)^{\circ m}$ and $\mathring{S}_m = (\mathring{S}_1)^{\circ m}$.)

Therefore to arrive at a sufficient condition for global existence, it suffices to study how the uniform-in-$\tau$ $L^\infty$ bounds and space-time $L^2$ bounds behave under rescaling. Observe first that
\begin{subequations}
\begin{gather}
	\label{eq:SmOcomm} \Omega_{ij} S_m \phi = S_m \Omega_{ij} \phi \\
	 \label{eq:SmDcomm} \partial S_m \phi = 2^{m} S_m \partial \phi \\
	 \label{eq:SmLcomm} L^i S_m \phi = S_m(L^i \phi) + (2^{m+1} - 2) S_m \partial_t\phi
\end{gather}
and
\begin{multline}\label{eq:SmKcomm}
(K + t) S_m \phi = 2^{-m} S_m (K+t) \phi + 4(1 - 2^{-m}) S_m( \sum \frac{x^i}{t} L^i \phi + \frac{\tau^2}{t} \partial_t\phi) \\
+ 2^{-m}(2^{m+1} - 2)^2 S_m \partial_t \phi + 2(1 - 2^{-m}) S_m\phi.\end{multline}
\end{subequations}

\begin{rmk}
	Notice that operators $\partial_t$, $L^i$, and $(K+t)$ are respectively homogeneous of degrees $-1$, $0$, and $1$ under space-time scaling of $\Real^{1,2}$. 
	So such estimates as above are expected. However, since our transformation $S_m$ is the conjugation of the scaling operation with a time-translation, and as $L^i$ and $(K+t)$ are not invariant under time-translation, we pick up \emph{lower order} correction terms. 
\end{rmk}

Also, on the set $\{t \geq \max(1 + |x|,2)\}$ one sees
\begin{equation}\label{eq:SmTauub}
	S_m \tau = 2^m \sqrt{ (t -2 + 2^{1-m})^2 - \abs{x}^2} < 2^{m} \tau. 
\end{equation}
Restricted to the set $\{ S_m t > \max(1 + |S_m x|, 2)\}$ we have
\[ S_m t > \max(1 + \abs{S_m x},2) \implies t > \max(\abs{x} + 2 - 2^{-m}, 2) \]
which after rewriting as $ t - \sqrt{t^2 - \tau^2} > 2 - 2^{-m}$ implies
\begin{equation}\label{eq:crucialtaubnd}
	\tau^2 + 4(1 - 2^{-m-1})^2 > 4(1 - 2^{-m-1}) t.
\end{equation}
Plugging \eqref{eq:crucialtaubnd} into the expression for $S_m \tau$ we see that, \emph{on the support of} $S_m \Phi_m$, we have 
\begin{align*}
	(S_m \tau)^2 & = 2^{2m} \left( \tau^2 - 4(1-2^{-m}) t + 4(1 - 2^{-m})^2\right) \\
		     & > 2^{m-1} \left( \frac{\tau^2}{1 - 2^{-m-1}} - 1 +  2^{-m}\right) \\
		     & > 2^{m-1} \left( \tau^2 - 1\right).
\end{align*}
Finally, noting that when restricted to the future $\{t \geq 2\}$, the bound \eqref{eq:crucialtaubnd} implies (when $m \geq 0$)
\[ \tau^2 > 4(1 - 2^{-2m - 2}) > 3 \]
we finally get the uniform bound
\begin{equation}\label{eq:SmTaulb}
	S_m \tau \geq \frac{2^{m/2}}{\sqrt{3}} \tau
\end{equation}
establishing that $\tau$ and $S_m\tau$ are comparable up to a factor depending on $m$. 

Now, fixing $k \geq 4$, the above computations imply the following lemma. 
Here let $Z^\ell$ be strings of length $\ell$ drawn from \eqref{eq:fullcommutators}, and implicitly we are only working within the forward light cone $\{ t > \max(1 + |x|,2)\}$. 

\begin{lem} \label{lem:comparisonscale}
	There exists a constant $C$ which depends on $k$ such that if $S_{m-1}\Phi_m$ satisfies the $L^\infty$ bounds for all $0 \leq \ell \leq k-2$
\[
				\begin{cases}
					\abs{\partial_t Z^\ell \Phi} \leq M \tau^{-1} \\
					\abs{L Z^\ell \Phi} \leq M \tau^{-1}  \\
					\abs{(K + t) Z^\ell \phi} \leq M \\
					\abs{\Phi} \leq M \ln(\tau) \tau^{-1} 
				\end{cases}
			\]
			and the $L^2$ bounds for all $0 \leq \ell \leq k$
			\[
				\int_{\{t > \max(1+ \abs{x},2)\}} \frac{ \ln(\tau)^2}{\tau^2\cosh(\rho)} \left[ (\partial_tZ^\ell \Phi)^2 + (L Z^\ell\Phi)^2 + \frac{1}{\tau^2} ([K + t]Z^\ell \Phi)^2 \right] ~\mathrm{dvol} \leq M^2,
			\]
			then $S_m\Phi_m$ satisfies the \emph{same} bounds with $M$ replaced by $CM$. 
\end{lem}

\begin{rmk}
	We are forced to deal with an enlarged set of commutators (more than just $L^i$), due to \eqref{eq:SmLcomm}. This means to estimate higher $L$ derivatives of $S_m\Phi_m$ we must also estimate higher $\partial_t$ derivative of $S_{m-1}\Phi_m$. Therefore it is not necessary that we fully embrace the full Poincar\'e algebra $\mathfrak{Z}$; we merely need to add control of the $\partial_t$ commutators into the argument.

	We note that including the whole set of commutators is no problem. As indicated in \thref{rmk:otherderag} using the arguments of \thref{rmk:otherder} the perturbation \thref{thm:pert} also hold using the larger set of commutators if the notion of dispersive solution in \thref{defn:dispsoln} is modified appropriately to include the full set of commutators from $\mathfrak{Z}$. 

	We note that this issue with the $\partial_t$ commutator arising can also be circumvented in a different manner. Instead of commuting with $\partial_t$, we can exploit an ``elliptic'' estimate for $\partial_t$. Noting that $\Phi_m$ solves a wave equation, one can control $\partial^2_{tt} S_{m-1}\Phi_m$ in terms of $\partial_t L S_{m-1}\Phi_m$, $LL S_{m-1}\Phi_m$, and the inhomogeneity/nonlinearity $\Box S_{m-1}\Phi_m$, none of which involves second order $t$ derivatives. In the present situation the first method seems simpler. 
\end{rmk}
\begin{proof}[of \thref{lem:comparisonscale}]
	We check here the terms involving the $(K+t)$ operator, which is the most complicated in view of \eqref{eq:SmKcomm}. The claim follows for the $\partial_t$ and $L$ terms straightforwardly in a similar manner from the computations \eqref{eq:SmDcomm}, \eqref{eq:SmLcomm}, \eqref{eq:SmTauub}, and \eqref{eq:SmTaulb}. 
	Without loss of generality we consider only the case where $\ell = 0$ (no commutators). Each commutation incurs some additional terms for which we have point-wise comparison by virtue of \eqref{eq:SmOcomm}, \eqref{eq:SmDcomm}, and \eqref{eq:SmLcomm}, and which only contribute additional constant factors. 

	Starting from \eqref{eq:SmKcomm} we can write
	\begin{multline*}
		|(K+t) S_{m} \Phi_m| \leq \frac12 \abs{ S_1 (K+t) S_{m-1}\Phi_m} + 2\sum \abs{ S_1  \frac{x^i}{t} L^i S_{m-1} \Phi_m } \\
		+ 2 \abs{ S_1 [(\frac{\tau^2}{t} + 1) \partial_t S_{m-1} \Phi_m]} + \abs{S_1 S_{m-1} \Phi_m}.
	\end{multline*}
	Let us first consider the point-wise estimate for $(K+t) S_{m} \Phi_m = (K+t) S_m S_{m-1} \Phi_m$. Applying the point-wise estimates for $S_{m-1}\Phi_m$ we get
	\[
		|(K+t) S_m \Phi_m | \leq \frac12 M + 4 M S_1 (\tau^{-1}) + 4 M S_1 (1 + \tau^{-1}) + M S_1( \ln(\tau) \tau^{-1}) .
	\]
	On the support of $S_{m-1}\Phi_m$ when $t \geq 2$, we have that $\tau \geq \sqrt{3}$, and $\ln(\tau) / \tau \leq 1 / e$. So we conclude pointwise
	\[ 
		|(K+t) S_m \Phi_m| \leq (\frac{8}{\sqrt{3}} + \frac{1}{e} + 4+ \frac12) M \approx 9.5 M.
	\]

	For the integral estimate, we note that since $S_m\Phi_m$ has, by construction and finite speed of propagation, support within the set $\{ S_1 t > \max(1 + |S_1 x|, 2)\}$, we have that on its support 
	\[ \sqrt{\frac23} \tau \leq S_1 \tau \leq 2 \tau.\]
	Similarly, we also have 
	\[  t \leq S_1 t \leq 2 t.\]
	This implies that 
	\[ \frac{\ln(\tau)^2}{\tau t} \leq \frac{16  S_1(\ln\tau)^2 }{S_1(\tau) S_1(t)}.\]
	So
	\begin{multline*}
		\int \frac{\ln(\tau)^2}{\tau^3 t} \left( [K+t] S_m \Phi_m\right)^2 ~\mathrm{dvol} \leq\\
		256 \int S_1\Bigg[ \frac{\ln(\tau)^2}{\tau^3 t}  \Big( \frac14 |(K+t)S_{m-1}\Phi_m|^2 + 4 |L S_{m-1} \Phi_m|^2 \\
		+ 4 \big(\frac{\tau^2}{t} + 1\big)^2 |\partial_t S_{m-1}\Phi_m|^2 + |S_{m-1}\Phi_m|^2 \Big) \Bigg] ~\mathrm{dvol}
	\end{multline*}
	Undoing the scaling (which we observe to map the set $\{ t \geq \max(1 + |x|, 2)\}$ into itself by construction) we get 
	\begin{multline*}
		\leq 32 \int  \frac{\ln(\tau)^2}{\tau^3 t}  \Big( \frac14 |(K+t)S_{m-1}\Phi_m|^2 + 4 |L S_{m-1} \Phi_m|^2 \\
		+ 4 \big(\frac{\tau^2}{t} + 1\big)^2 |\partial_t S_{m-1}\Phi_m|^2 + |S_{m-1}\Phi_m|^2 \Big) ~\mathrm{dvol}.
	\end{multline*}
	The integral of the first three terms can be bounded by $512 M^2$ using the assumed $L^2$ bound. 
	(Notice that $\frac{1}{\tau^2} (\tau^2 / t + 1)^2 \leq 4$.)
	For the final term in the integrand, we can use our log weighted Hardy inequality \thref{lem:hardy:2} together with the support property of $S_{m-1}\Phi_m$ to get 
\begin{multline*} \int \frac{\ln(\tau)^2}{\tau^4 \cosh\rho} |S_{m-1}\Phi_m|^2 ~\mathrm{dvol} \lesssim \int \frac{\ln(\tau)^3}{\tau^4 \cosh\rho} |L S_{m-1}\Phi_m|^2 ~\mathrm{dvol} \\
\lesssim \int \frac{\ln(\tau)^2}{\tau^2 \cosh\rho} |LS_{m-1}\Phi_m|^2 ~\mathrm{dvol} \lesssim M^2
\end{multline*}
	as needed. (The implicit constant here given by the Hardy inequality.)
\end{proof}

A simple induction argument using \thref{thm:pert} yields then the following sufficient condition for global existence:

\begin{cor}
	Let $C$ be the constant of \thref{lem:comparisonscale} and let $\lambda$ be the constant from \thref{thm:pert}. Then if the initial data $(\mathring{\phi}, \mathring{\varphi})$ is such that 
	\[ \sup_{m \geq 0} \left( \norm[H^{k+1}]{\mathring{S}_m \mathring{\phi}_m} + 2^m \norm[H^k]{\mathring{S}_m \mathring{\varphi}_m} \right) \exp\Bigg( \lambda \Big(1+ \frac{C^m - 1}{C - 1} \Big)^2 \Bigg) \leq 1\]
	there exists a global solution. 
\end{cor}

We can rephrase the smallness and localization of the initial data in terms of a weighted Sobolev space that requires super-exponential decay of data. 
\begin{cor}
	There exists constants $\kappa, \nu > 0$ such that whenever the initial data $(\mathring{\phi}, \mathring{\varphi})$ is sufficient small in the weighted Sobolev space 
	\[ H^{k+1}(\mathbb{R}^2, \exp(\nu(1 + |x|^2)^\kappa)\D x) \times H^k(\mathbb{R}^2, \exp(\nu(1 + |x|^2)^\kappa)\D x)\]
	there exists a global solution to the initial value problem for \eqref{eq:WMmain}. 
\end{cor}

\printbibliography
\end{document}